\documentclass[11pt]{amsart}

\usepackage{amsmath}
\usepackage{paralist}
\usepackage{graphics}
\usepackage{epsfig}
\usepackage{graphicx}
\usepackage{epstopdf}
\usepackage[colorlinks=true]{hyperref}
\usepackage{xcolor}

\hypersetup{
	colorlinks=true,
	citecolor=red,
	linkcolor=red,
	urlcolor=orange,
}
\usepackage[margin=1in]{geometry}

\usepackage{microtype}
\usepackage[utf8]{inputenc}
\usepackage[T1]{fontenc}
\usepackage{amssymb, amsfonts}
\usepackage{subcaption}
\usepackage{algorithm}
\usepackage{algorithmic}
\usepackage{enumitem}
\usepackage{tabularx}
\usepackage[english]{babel}
\newtheorem{theorem}{Theorem}[section]
\newtheorem{corollary}[theorem]{Corollary}
\newtheorem{lemma}[theorem]{Lemma}
\newtheorem{proposition}[theorem]{Proposition}

\newtheorem{definition}[theorem]{Definition}
\newtheorem{remark}[theorem]{Remark}
\newtheorem{assumption}[theorem]{Assumption}

\newcommand{\Var}{\mathrm{Var}}
\newcommand{\Cov}{\mathrm{Cov}}
\newcommand{\Corr}{\mathrm{Corr}}

\newcommand{\Law}{\mathcal{L}} 

\title[Mean-Field Analysis of Rating Systems]{Mean-Field Analysis and Optimal Control of a Dynamic Rating and Matchmaking System}

\author[Wataru Nozawa]{Wataru Nozawa$^1$}

\address{$^1$ Department of Economics, Fukuoka University, \newline
8-19-1 Nanakuma, Jonan-ku, Fukuoka 814-0180, Japan}

\email{nozawa@fukuoka-u.ac.jp}



\begin{document}

\maketitle

\begin{abstract}
Large-scale competitive platforms are interacting multi-agent systems in which latent skills drift over time and pairwise interactions are shaped by matchmaking.
We study a controlled rating dynamics in the mean-field limit and derive a kinetic description for the joint evolution of skills and ratings.
In the Gaussian regime, we prove an exact moment closure and obtain a low-dimensional deterministic state dynamics for rating accuracy.
This yields three main insights.
First, skill drift imposes an intrinsic ceiling on long-run accuracy (the ``Red Queen'' effect).
Second, with period-by-period scale control, the information content of interactions satisfies an invariance principle: under signal-matched scaling, the one-step accuracy transition is independent of matchmaking intensity.
Third, the optimal platform policy separates: filtering is implemented by a greedy choice of the gain and rating scale, while matchmaking reduces to a static trade-off between match utility and sorting costs.
\end{abstract}

\section{Introduction}
\label{sec:intro}

Large-scale competitive platforms---online games, crowdsourcing markets, and digital labor exchanges---continuously infer and display \emph{ratings} for a massive population of users whose latent abilities are unobserved and evolve over time.
A central operational challenge is therefore \emph{tracking}: the platform must update ratings quickly enough to follow drifting skills, while simultaneously organizing pairwise interactions (matchmaking) to balance fairness, engagement, and latency.

From a modeling viewpoint, such platforms are naturally described as interacting-agent systems with control acting through both (i) \emph{how signals are processed} (the rating update rule) and (ii) \emph{who interacts with whom} (the interaction kernel).
This leads to a mean-field control problem: the controller does not only estimate the state distribution but steers the population dynamics by choosing the update parameters and the matching assortativity.
While individual-level rating algorithms---most notably Elo \cite{Elo1978} and Bayesian/state-space extensions such as Glicko \cite{Glickman1999,Glickman2001} and TrueSkill \cite{Herbrich2007}---are widely deployed, a rigorous understanding of their \emph{collective} dynamics under \emph{platform design} objectives remains comparatively limited.

This paper develops a kinetic/mean-field framework for a stylized but analytically tractable design problem.
Agents have latent skills that follow a stationary drift process (an AR(1)/discrete OU dynamics with persistence $\lambda$), and the platform observes noisy pairwise outcomes generated by skill differences.
At each period $t$, the platform chooses a control triple
\[
u_t=(K_t,\eta_t,\sigma_t),
\]
where $K_t$ is the update gain, $\eta_t$ parametrizes assortativity of the interaction kernel, and $\sigma_t\ge 0$ is the \emph{rating scale applied in period $t$} (implemented operationally by a deterministic normalization of the cross-sectional rating dispersion).
This ``period-by-period scale control'' directly matches the way the main optimal-scaling policy is stated and implemented in the numerics (e.g., $\sigma_t=r_t$).

Our first step is to connect the microscopic particle system to a macroscopic kinetic description.
We prove a mean-field limit and obtain a McKean--Vlasov evolution for the joint law of skill and rating.
Under Gaussian initial conditions and linear update rules, the Gaussian family is invariant and the mean-field dynamics admit an \emph{exact finite-dimensional reduction}: the evolution closes on second moments, yielding a one-dimensional deterministic state variable capturing rating accuracy (the correlation $r_t=\Corr(\rho_t,X_t)$).

This tractability makes it possible to solve the platform-design problem analytically and reveals a sharp structural message.
Once the platform uses the scale control $\sigma_t$ to properly normalize the rating dispersion, the information content of each match becomes \emph{invariant} to how assortative matching is.
Consequently, filtering (choosing $(K_t,\sigma_t)$ to maximize one-step accuracy) and matchmaking (choosing $\eta_t$ to maximize instantaneous utility net of sorting costs) separate.
In the presence of skill drift, we also characterize an entropic ceiling: even under optimal filtering, the long-run accuracy is strictly bounded away from one (the ``Red Queen'' effect), and optimal matchmaking exhibits a cost-driven transition between a disordered regime (near-random matching) and an ordered regime (strict sorting).

All proofs and detailed derivations are provided in the Appendix.
In particular, an \emph{optional} continuous-time diffusion approximation is relegated to Appendix~\ref{app:ct_limit}; it is not used in the discrete-time control results nor in the numerical experiments.

\subsection{Related Literature}
Our paper connects three strands of work: (i) kinetic and mean-field descriptions of interacting-agent systems, including rating dynamics; (ii) statistical models and online algorithms for paired comparisons; and (iii) mean-field control problems with control-dependent interaction kernels.

\subsubsection*{Kinetic/mean-field models and rating dynamics.}
Kinetic and mean-field limits have been extensively studied in socio-economic modeling and collective dynamics; see, e.g., \cite{Cordier2005,During2008,Toscani2006,Carrillo2010,HaTadmor2008}.
In the context of ratings, \cite{JabinJunka2015} derives a kinetic description from Elo-type microscopic updates under exogenous interactions.
Related diffusion/Fokker--Planck descriptions and long-time behavior are analyzed in \cite{During2019,DuringEvansWolfram2024}.
Probabilistic and Markov-chain approaches provide convergence guarantees under simplified settings \cite{Aldous2017,Olesker-Taylor2024,CortezTossounian2024}.
Our contribution differs by explicitly incorporating \emph{time-varying skills} and, crucially, by studying a \emph{platform design} problem in which the interaction kernel (assortativity of matching) is itself chosen optimally.

\subsubsection*{Paired-comparison models and rating algorithms.}
Our model builds on classical paired-comparison foundations \cite{Bradley1952} and Elo-style updates \cite{Elo1978}.
Bayesian systems such as Glicko \cite{Glickman1999,Glickman2001} and TrueSkill \cite{Herbrich2007} represent uncertainty about skill and are designed for non-stationary environments, typically via approximations for scalability.
In contrast, our mean-field approach yields an analytically tractable population-level filtering theory: within the Gaussian regime, the mean-field dynamics close exactly on moments, providing an explicit state evolution for accuracy.

\subsubsection*{Mean-field control with kernel choice.}
Our platform problem lies within mean-field control and mean-field games \cite{LasryLions2007,Huang2006,Bensoussan2013,Carmona2018}, with the distinctive feature that control acts through the \emph{interaction kernel}.
Such problems are typically characterized by HJB equations on spaces of measures.
Here, the linear-Gaussian information structure yields a strong separation: after optimal scaling, learning dynamics become independent of assortativity, reducing the dynamic problem to tractable subproblems.

\subsubsection*{Matchmaking and platform evidence.}
Our analysis also relates to empirical and algorithmic work on matchmaking and engagement in online platforms \cite{Activision2024,Kim2024,Claypool2015,Yuval2019}.
While that literature focuses on performance and user experience, our kinetic framework provides explicit macroscopic predictions---including invariance and a cost-driven transition between ordered and disordered matching regimes---that can serve as a mathematical foundation for platform design questions.

\subsection{Summary of Contributions}
We develop a kinetic/mean-field framework for the analysis and design of large-scale rating and matchmaking systems with drifting skills.
The platform controls both the update rule and the interaction kernel via $u_t=(K_t,\eta_t,\sigma_t)$.
In the Gaussian regime (Gaussian initial conditions), we obtain an exact finite-dimensional reduction and an analytically tractable mean-field control problem.
Our main contributions are:

\begin{itemize}
\item \textbf{Mean-field limit and exact finite-dimensional reduction.}
We derive the mean-field evolution of the joint law of latent skills and observable ratings as a McKean--Vlasov system.
Under Gaussian initial conditions, the Gaussian family is invariant and the moment dynamics close exactly (Proposition~\ref{prop:gaussianity}), yielding a low-dimensional deterministic state dynamics (Theorem~\ref{thm:transition_function}).

\item \textbf{Invariance principle for information acquisition under scale control.}
We prove that when the platform chooses the period-$t$ rating scale in a signal-matched way (operationally, $\sigma_t=r_t$), the one-step accuracy transition becomes independent of assortativity $\eta_t$ (Theorem~\ref{thm:invariance}).
Thus, after proper normalization, assortative matching affects \emph{instantaneous utility} but not \emph{information flow}.

\item \textbf{Separation of filtering and matching.}
Building on invariance, we establish a separation result (Corollary~\ref{cor:separation}): optimal filtering controls $(K_t,\sigma_t)$ are obtained by a greedy step that maximizes next-period accuracy, while optimal matching $\eta_t$ solves a static instantaneous trade-off between match utility and sorting costs.

\item \textbf{Entropic ceiling under drift and a phase transition in optimal matching.}
With drifting skills, we characterize a unique globally stable steady state and show that long-run accuracy is strictly bounded by $\lambda$ (Proposition~\ref{prop:equilibrium}; ``Red Queen'' effect).
Moreover, with a queueing-type sorting cost, the optimal matching policy exhibits a cost-driven transition between a disordered regime (near-random matching) and an ordered regime (strict sorting) (Proposition~\ref{prop:phase_transition}).

\item \textbf{Finite-population validation.}
Particle simulations validate the mean-field predictions, quantify finite-$N$ deviations, and illustrate the invariance-induced data collapse under optimal scaling.
\end{itemize}

\section{Kinetic Modeling and Mean-Field Dynamics}
\label{sec:model}

We consider a large system of $N$ competing agents. The state of the system at time $t \in \mathbb{N}$ is described by the collection of microscopic states $\{( \rho_{i,t}, X_{i,t} )\}_{i=1}^N$, where $\rho_{i,t} \in \mathbb{R}$ denotes the latent skill and $X_{i,t} \in \mathbb{R}$ denotes the observable rating of agent $i$.
For notational convenience, we normalize the initial ratings to zero, i.e., $X_{i,0}=0$ for all $i$.
This makes the initial rating distribution degenerate, so that $\sigma_0=0$ and the correlation-based accuracy $r_t=\mathrm{Corr}(\rho_t,X_t)$ is not defined at $t=0$; we set $r_0:=0$ by convention and apply the correlation-based state description from the first non-degenerate step (typically $t\ge 1$).
Our goal is to derive the macroscopic laws governing the evolution of the joint probability density $f_t(\rho, x)$ in the mean-field limit $N \to \infty$, establishing the dynamic constraints for the control problem in Section \ref{sec:analysis}.

\subsection{Microscopic Particle Dynamics}
\label{sec:micro_dynamics}

The particle system is driven by two mechanisms: intrinsic skill drift and binary interactions organized by the platform.
Throughout, $i\in\{1,\dots,N\}$ indexes agents, $\rho_{i,t}$ denotes latent skill, and $X_{i,t}$ denotes the observable rating.

\paragraph{Intrinsic Skill Dynamics (discrete-time OU).}
Latent skill follows a discrete-time Ornstein--Uhlenbeck dynamics:
\begin{equation}
\label{eq:micro_skill}
    \rho_{i, t+1} = \lambda \rho_{i, t} + \sqrt{1 - \lambda^2}\,\xi_{i, t},
    \qquad
    \xi_{i, t} \sim \mathcal{N}(0, 1),
\end{equation}
where $\lambda\in[0,1]$ is the persistence parameter. This ensures the stationary skill distribution is standard Gaussian
$\mathcal{M}(\rho)=(2\pi)^{-1/2}e^{-\rho^2/2}$.

\paragraph{Interaction Mechanism (scaling, matching, and update).}
The platform specifies three sequences
\[
\{K_t\}_{t\ge 0}\subset(0,\infty),
\qquad
\{\sigma_t\}_{t\ge 0}\subset(0,\infty),
\qquad
\{\eta_t\}_{t\ge 0}\subset[0,1),
\]
which jointly determine the rating dispersion, the matchmaking intensity, and the rating update rule.
Given the cross-sectional rating vector at the beginning of period $t$, the period-$t$ interaction proceeds in three steps.

\medskip
\noindent\textbf{Step 1 (dispersion scaling).}
At the beginning of each period $t$, the platform rescales the ratings to achieve the target dispersion $\sigma_t^2$.
Concretely, given pre-scaled ratings $\widetilde X_{i,t}$ (coming from the previous update), define
\begin{equation}
\label{eq:micro_scaling}
    X_{i,t} = L_t\,\widetilde X_{i,t},
    \qquad
    L_t := \frac{\sigma_t}{\Lambda_t},
    \qquad
    \Lambda_t^2 := \Var(\widetilde X_{i,t}),
\end{equation}
so that $\Var(X_{i,t})=\sigma_t^2$.%
\footnote{One may additionally recenter by subtracting the cross-sectional mean; this does not affect the correlation-based state variable used below.}

\medskip
\noindent\textbf{Step 2 (correlated matching with intensity $\eta_t$).}
Using the scaled ratings $\{X_{i,t}\}$, the platform forms pairs so that matched ratings are positively correlated with strength $\eta_t\in[0,1)$.
The matching procedure is implemented in a finite population as follows.

At time $t$, the population is randomly split into two groups of equal size (Group 1 and Group 2), each preserving the marginal distribution of $X_t$.
For each player $i$ in Group 1, construct a \emph{matching score}
\begin{equation}
\label{eq:micro_match_score}
    Y_{i,t} := \eta_t X_{i,t} + \sqrt{1-\eta_t^2}\,v_{i,t},
    \qquad
    v_{i,t}\sim\mathcal{N}(0,\sigma_t^2)\ \text{independent}.
\end{equation}
Then sort Group 1 by the scores $\{Y_{i,t}\}$ and sort Group 2 by ratings $\{X_{j,t}\}$, and match players by rank (the $k$-th in the $Y$-ordering is paired with the $k$-th in the $X$-ordering).
Equivalently, each player $i$ in Group 1 is matched to a player $j$ in Group 2 whose rating $X_{j,t}$ is closest (in quantile rank) to the target value $Y_{i,t}$.

In the Gaussian regime considered below, $X_{i,t}\sim\mathcal{N}(0,\sigma_t^2)$ and $v_{i,t}$ is independent with the same variance, so the score $Y_{i,t}$ has the same marginal distribution:
\[
Y_{i,t}\sim\mathcal{N}(0,\sigma_t^2).
\]
Hence rank matching is well-defined and produces a coupling of $(X_{i,t},X_{j,t})$ that approximates the idealized mean-field coupling $(X_{i,t},Y_{i,t})$.
In particular, the induced correlation satisfies
\begin{equation}
\label{eq:micro_match_corr}
    \Corr(X_{i,t},X_{j,t}) \approx \Corr(X_{i,t},Y_{i,t}) = \eta_t,
\end{equation}
with the approximation error vanishing as $N\to\infty$ (the mean-field limit).\footnote{In a finite population, exact equality $X_{j,t}=Y_{i,t}$ is typically impossible; rank (quantile) matching provides a canonical discretization that converges to the continuum coupling.}
When $\eta_t=0$, the score is independent of $X_{i,t}$ and matching is effectively random; as $\eta_t\uparrow 1$, matching becomes increasingly assortative, approaching perfectly like-with-like pairing.

\medskip
\noindent\textbf{Step 3 (rating update).}
Given a matched pair $(i,j)$, the platform observes a noisy outcome
\begin{equation}
\label{eq:micro_outcome}
    S_{ij,t} = \rho_{i,t}-\rho_{j,t}+\omega_{ij,t},
    \qquad
    \omega_{ij,t}\sim\mathcal{N}(0,\beta^2),
\end{equation}
and updates \emph{pre-scaled} next-period ratings by the linear rule
\begin{equation}
\label{eq:micro_update_pre}
\begin{cases}
\widetilde X_{i,t+1}
    = X_{i,t} + K_t\Big(S_{ij,t}-(X_{i,t}-X_{j,t})\Big),\\[0.3em]
\widetilde X_{j,t+1}
    = X_{j,t} + K_t\Big(-S_{ij,t}-(X_{j,t}-X_{i,t})\Big).
\end{cases}
\end{equation}
The next period starts by applying Step 1 with $\sigma_{t+1}$, i.e., $X_{t+1}=L_{t+1}\widetilde X_{t+1}$, and the procedure repeats.

\medskip
This formulation makes explicit that $\{\sigma_t\}$ governs the rating dispersion \emph{period by period}, while $\{\eta_t\}$ determines the correlation structure of matched ratings and $\{K_t\}$ controls the strength of the update.
The schematic of this interaction cycle is illustrated in Figure~\ref{fig:interaction_cycle}.
\begin{figure}[htbp]
    \centering
    \includegraphics[width=0.95\textwidth]{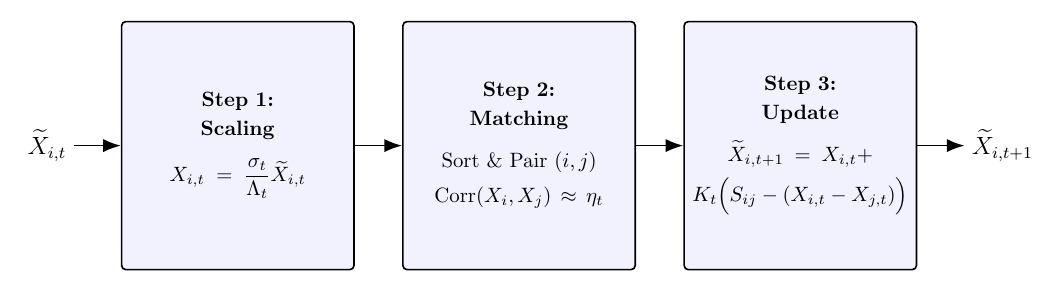}
    \caption{Schematic of the period-$t$ interaction cycle. The platform sequentially applies scaling ($\sigma_t$), matching ($\eta_t$), and updating ($K_t$) to evolve the rating distribution.}
    \label{fig:interaction_cycle}
\end{figure}
The role of scaling in preserving tractability (Gaussian closure and moment dynamics) is formalized in the subsequent mean-field analysis (Theorem~\ref{thm:transition_function}).

\subsection{The Kinetic Equation}
In the mean-field limit $N \to \infty$, the empirical measure converges to a probability density $f_t(\rho, x)$. The evolution of $f_t$ is governed by a Boltzmann-type kinetic equation, involving a transport-diffusion operator $\mathcal{T}$ (skill drift) and a collision operator $\mathcal{Q}$ (rating update).

\paragraph{Matchmaking Kernel.}
For a given matching intensity $\eta_t\in[0,1)$, we define the \emph{matching kernel}
$\mathcal{K}_{\eta_t}(x,x_*)$ as the conditional density of an opponent rating $x_*$ given own rating $x$:
\begin{equation}
\label{eq:matching_kernel}
\mathcal{K}_{\eta_t}(x,x_*)=
\frac{1}{\sqrt{2\pi (1-\eta_t^2)\sigma_t^2}}
\exp\!\left(
-\frac{(x_*-\eta_t x)^2}{2(1-\eta_t^2)\sigma_t^2}
\right),
\end{equation}
where $\sigma_t^2:=\iint x^2 f_t(\rho,x)\,d\rho\,dx=\int x^2 g_t(x)\,dx$ is the current rating second moment
(variance under our zero-mean normalization), and $g_t(x):=\int f_t(d\rho,dx)$ denotes the marginal density of ratings.
We interpret $\mathcal{K}_{\eta_t}(x,\cdot)$ as a \emph{Markov kernel} on ratings: conditional on an agent
with rating $x$, the opponent's rating $X_{*,t}$ is drawn according to $\mathcal{K}_{\eta_t}(x,dx_*)$ as in
\eqref{eq:matching_kernel}.
\emph{Degenerate case.} When $\sigma_t=0$ (in particular, at $t=0$ under our normalization $X_{i,0}=0$),
the Gaussian density representation \eqref{eq:matching_kernel} is not available. In this case we extend
$\mathcal{K}_{\eta_t}(x,\cdot)$ as a possibly singular Markov kernel by setting
$\mathcal{K}_{\eta_t}(x,dx_*)=\delta_{\eta_t x}(dx_*)$, so that $\mathcal{K}_{\eta_0}(0,dx_*)=\delta_0(dx_*)$.
With this extension, the weak formulation \eqref{eq:boltzmann_weak} remains well-defined even when $\sigma_t=0$.
Conditional on $X_{*,t}=x_*$, the opponent's latent skill is drawn from the conditional distribution
$f_t(d\rho_* \mid x_*):=f_t(d\rho_*,dx_*)/g_t(x_*)$ (whenever $g_t(x_*)>0$).
Under the Gaussian closure, $g_t$ is centered Gaussian with variance $\sigma_t^2$, and the kernel in
\eqref{eq:matching_kernel} induces $\mathrm{Corr}(X_t,X_{*,t})=\eta_t$ while preserving the marginal.

\begin{remark}[Admissible range of the matching intensity]
\label{rem:eta_range}
Throughout the paper we restrict the platform's sorting intensity to $\eta_t\in[0,1)$.
The Gaussian \emph{density} representation \eqref{eq:matching_kernel} is non-degenerate only for $|\eta_t|<1$.
The extreme choice $\eta_t=1$ (perfect sorting) is operationally meaningful, and it can be incorporated by extending the kernel as
$\mathcal{K}_{1}(x,dx')=\delta_x(dx')$ (a Dirac kernel). We nevertheless exclude $\eta_t=1$ because our analysis relies on density-based expressions and because the sorting cost satisfies $\lim_{\eta\uparrow 1}C(\eta)=+\infty$ (Section~\ref{subsec:control_problem}).
Finally, negative values $\eta_t\in(-1,0)$ correspond to reverse sorting. Such policies run counter to skill-balancing objectives and are not considered by the platform; accordingly we impose $\eta_t\ge0$ without loss for the design problem studied here.
\end{remark}

\paragraph{Weak Formulation.}
The evolution of $f_{t+1}$ is defined in weak form: for any test function $\phi(\rho,x)$,
\begin{equation}
\label{eq:boltzmann_weak}
\begin{aligned}
\int \phi(\rho,x)\, f_{t+1}(d\rho,dx)
&= \iint f_t(d\rho,dx)
\int \mathcal{K}_{\eta_t}(x,dx_*)
\int f_t(d\rho_* \mid x_*) \\
&\qquad\qquad \times
\mathbb{E}\!\left[\,\phi(\rho',X')\,\middle|\,\rho,x,\rho_*,x_* \right].
\end{aligned}
\end{equation}
Equivalently, writing $g_t(x_*):=\int f_t(d\rho_*,dx_*)$, we have
\[
\int \phi\, df_{t+1}
=
\iiint\!\!\int
\mathbb{E}\!\left[\,\phi(\rho',X')\,\middle|\,\rho,x,\rho_*,x_* \right]\,
\mathcal{K}_{\eta_t}(x,x_*)\,
f_t(\rho,x)\,
\frac{f_t(\rho_*,x_*)}{g_t(x_*)}\,
d\rho\,dx\,d\rho_*\,dx_* .
\]

\subsection{Macroscopic Reduction and Gaussian Closure}
Solving the integro-differential equation \eqref{eq:boltzmann_weak} on the space of measures is computationally intractable for control purposes. However, the Linear-Gaussian structure of our system admits an exact finite-dimensional reduction. \sloppy

\begin{proposition}[Preservation of Gaussianity]
\label{prop:gaussianity}
If the initial density $f_0(\rho, x)$ is a zero-mean bivariate Gaussian, then the solution $f_t(\rho, x)$ to the kinetic equation remains a zero-mean bivariate Gaussian for all $t \ge 0$.
\end{proposition}

\begin{proof}
Assume $f_t(\rho,x)$ is a zero-mean (possibly degenerate) bivariate Gaussian.
Let $\sigma_t^2:=\Var(\bar X_t)$ and let $m_t:=\Cov(\bar\rho_t,\bar X_t)$.
If $\sigma_t=0$ (so $\bar X_t\equiv 0$ a.s.), we may simply draw $\bar X_t'\equiv 0$ and take
$\bar\rho_t'\sim\mathcal N(0,1)$ independent of everything below.
Henceforth assume $\sigma_t>0$ and set $r_t:=m_t/\sigma_t=\Corr(\bar\rho_t,\bar X_t)$.
We construct the mean-field interaction using the ``shadow opponent'' coupling (see Appendix~\ref{app:meanfield}).
Given $\bar X_t$, draw
\[
\bar X_t'=\eta_t \bar X_t+\sqrt{1-\eta_t^2}\,Z_t,\qquad
Z_t\sim \mathcal N(0,\sigma_t^2),\quad Z_t\perp \bar X_t,
\]
so that $\Corr(\bar X_t,\bar X_t')=\eta_t$ and $\bar X_t'$ is Gaussian with variance $\sigma_t^2$.
Under the Gaussian hypothesis, the conditional law of skill given rating is linear; hence we may represent the opponent skill as
\[
\bar\rho_t'=\frac{r_t}{\sigma_t}\bar X_t'+\sqrt{1-r_t^2}\,\zeta_t,\qquad
\zeta_t\sim \mathcal N(0,1),
\]
independent of $(\bar X_t,\bar X_t',Z_t)$, which enforces $\Cov(\bar\rho_t',\bar X_t')=r_t\sigma_t$ and $\Var(\bar\rho_t')=1$.
Together with the AR(1) skill update \eqref{eq:micro_skill} and the (mean-field) linear rating update \eqref{eq:mckean_vlasov}, we obtain $(\bar\rho_{t+1},\bar X_{t+1})$ as an affine function of a jointly Gaussian vector
$(\bar\rho_t,\bar X_t,Z_t,\zeta_t,\xi_t,\omega_t)$.
Therefore $(\bar\rho_{t+1},\bar X_{t+1})$ is Gaussian and has mean zero; scaling by the deterministic factor $L_{t+1}$ preserves Gaussianity.
By induction, $f_t$ remains Gaussian for all $t\ge0$.
\end{proof}

This property implies that the infinite-dimensional dynamics of $f_t$ project exactly onto the dynamics of its second moments.
Since the skill variance is fixed at unity, the state is fully characterized by the rating variance $\sigma_t^2$ and the rating accuracy (correlation)
\[
r_t := \sigma_t^{-1} \int \rho x\, f_t(\rho,x)\, d\rho dx , \qquad (t\ge 1),
\]
whenever $\sigma_t>0$.
Under our normalization $X_{i,0}=0$, the initial rating distribution is degenerate and $\sigma_0=0$, so the correlation is not defined at $t=0$.
For notational convenience we set $r_0:=0$, and the analysis of rating accuracy starts from $t\ge 1$ (where $\sigma_t>0$ along any nontrivial trajectory and in the simulations).

We now state the main dynamical constraint that will be used in the optimization problem.

\begin{theorem}[State Transition Function]
\label{thm:transition_function}
The evolution of the rating accuracy $r_t$ is governed by the deterministic map $r_{t+1} = \Psi(r_t, K_t, \eta_t, \sigma_t)$:
\begin{equation}
\label{eq:main_dynamics}
    r_{t+1} = \lambda \cdot \frac{r_t \sigma_t (1 - K_t (1 - \eta_t)) + K_t (1 - \eta_t r_t^2)}{ \sqrt{ \Lambda(r_t, K_t, \eta_t, \sigma_t)^2 } }.
\end{equation}
Here, $\Lambda^2$ represents the pre-scaling variance of the ratings:
\begin{align}
\label{eq:variance_term}
    \Lambda^2 &= \sigma_t^2 \left[ (1-K_t)^2 + K_t^2 + 2K_t(1-K_t)\eta_t \right] \nonumber \\
    &\quad + K_t^2 \left[ \beta^2 + 2(1 - \eta_t r_t^2) \right] \nonumber \\
    &\quad + 2K_t(1-K_t)(1-\eta_t) r_t \sigma_t - 2K_t^2 (1-\eta_t) r_t \sigma_t.
\end{align}
The explicit dependence on $\sigma_t$ indicates that the optimal policy must jointly control the gain and the scale.
\end{theorem}

\begin{proof}
The proof follows from testing \eqref{eq:boltzmann_weak} with $\phi(x) = x^2$ and $\phi(\rho, x) = \rho x$. The expectation terms involve Gaussian integrals over the kernel $\mathcal{K}_\eta$, which yield the algebraic expressions in \eqref{eq:variance_term}. See Appendix \ref{app:variance_derivation} for details.
\end{proof}

\paragraph{Analysis of the Transition Map $\Psi$}
The complexity of the transition function $\Psi$ (and thus the reduced dynamics) is primarily contained within the pre-scaling variance term $\Lambda^2$, which represents the total uncertainty injected into the rating $X$ during a single update step. $\Lambda^2$ is an aggregate of three distinct, independent sources of variance in the unscaled rating $\tilde{X}_{t+1}$:

\begin{enumerate}
    \item \emph{Retention Term}: The variance retained from the previous rating state, $\Var((1-K_t)X_t + K_t X'_t)$, which is modulated by the match correlation $\eta_t$.
    \item \emph{Skill Mismatch Term}: The variance injected by the latent skill difference and mismatch, $\Var(K_t(\rho_t - \rho'_t))$.
    \item \emph{Observation Noise Term}: The variance from the inherent match outcome uncertainty, $\Var(K_t \omega_t)$.
\end{enumerate}

The algebraic expression for $\Lambda^2$ is intricate due to the cross-covariance terms that couple the current rating $X_t$ with the opponent's rating $X'_t$ and skill $\rho'_t$. A rigorous, stepwise derivation detailing the contribution of each component is provided in Appendix \ref{app:variance_derivation}, confirming the explicit form presented in this theorem.

\section{Mean-Field Control and Asymptotic Analysis}
\label{sec:analysis}

Having derived the kinetic description of the system, we now address the \textit{inverse problem}: designing the interaction rules to optimize the system's performance. From a mathematical perspective, this constitutes a Mean-Field Control (MFC) problem. The central planner (platform) seeks to control the evolution of the probability density $f_t$ by adjusting the interaction kernel $\mathcal{K}_{\eta}$ and the post-collision redistribution.

\subsection{The Optimal Control Problem}
\label{subsec:control_problem}
We study the design problem faced by a competitive platform that repeatedly matches agents and updates
ratings. By Gaussian closure, the platform's informational state is fully summarized by the scalar
\emph{rating accuracy} $r_t\in[0,1]$, the correlation between true skill and rating. Hence the original
mean-field control problem reduces to a finite-dimensional control problem on moments with state
variable $r_t$.

\subsubsection{Utility of Match Quality}
The primary goal of a rating system is to facilitate fair competitions. In our mean-field model,
fairness is quantified by the correlation of skills in a matched pair:
\[
\mathbb{E}[\rho\rho'] = \eta_t r_t^2,
\]
where $\eta_t\in[0,1)$ measures the sorting intensity of the matching kernel and $r_t$ measures how
informative ratings are about skills. Empirical evidence from industry and academic studies indicates
that stronger skill balance improves user experience and retention \cite{Activision2024,Kim2024}.
We therefore interpret $\eta_t r_t^2$ as the flow utility from match quality.

\subsubsection{Cost of Sorting (Latency)}
Stricter sorting reduces the set of eligible opponents and increases expected waiting time. This trade-off
is well documented for competitive online games \cite{Claypool2015,Yuval2019}. We model this friction as a
convex queueing cost $C(\eta)$ satisfying
\[
C:[0,1)\to\mathbb{R}_+ \text{ is convex, nondecreasing, and } C(0)=0,\qquad
\lim_{\eta\uparrow 1} C(\eta)=+\infty,
\]
capturing the asymptotic explosion of waiting times under nearly perfect sorting.

\subsubsection{Controls, State Dynamics, and Objective}
We now define \emph{admissibility} of control processes; see the definition below.

\begin{definition}[Admissible policies]
\label{def:admissible}
A policy $\pi=\{(K_t,\sigma_t,\eta_t)\}_{t\ge0}$ is \emph{admissible} if $\sigma_0=0$ and
\[
K_t>0,\qquad \sigma_{t+1}>0,\qquad \eta_t\in[0,1)\qquad\text{for all }t\ge0.
\]
Given $r_0\in[0,1)$, the induced state sequence $\{r_t\}$ is defined recursively by the exact transition
map (Theorem~\ref{thm:transition_function}):
\begin{equation}
\label{eq:reduced_dynamics_31}
r_{t+1} = \Psi(r_t,K_t,\eta_t,\sigma_t).
\end{equation}
\end{definition}

The platform maximizes discounted net welfare:
\begin{equation}
\label{eq:mfc_problem}
J(\pi)=\sum_{t=0}^{\infty}\delta^t\Bigl[\eta_t r_t^2 - C(\eta_t)\Bigr],
\qquad \delta\in(0,1).
\end{equation}

\paragraph{Timing convention.}
At each time $t$, the platform chooses $(K_t,\sigma_t,\eta_t)$.
After the rating update, it rescales the post-update ratings so that the resulting dispersion equals
the chosen target $\sigma_t$ (which becomes the dispersion that enters the kernel for the next matching step).

\begin{remark}[Why non-compactness matters and how we address it]
\label{rem:welldefined_noncompact}
Because $(K_t,\sigma_t)\in(0,\infty)^2$, the admissible set is non-compact. The objective in
\eqref{eq:mfc_problem} is nevertheless well-defined and finite for any admissible policy since
$\eta_t r_t^2\le 1$ and $C(\eta_t)\ge 0$, hence $J(\pi)\le \sum_{t\ge0}\delta^t<\infty$.
The nontrivial issue is whether the optimal choices of $(K_t,\sigma_t)$ are attained despite the
non-compact domain. The next subsection resolves this by showing that the one-step filtering problem
has a unique interior maximizer.
\end{remark}

\subsection{One-Step Filtering, Fisher Information, and the Invariance Principle}
\label{subsec:invariance}
We now isolate the part of the control problem that governs \emph{information acquisition}. Fix the
current accuracy $r\in(0,1)$ and sorting intensity $\eta\in[0,1)$. Consider the one-step filtering problem:
choose $(K,\sigma)\in(0,\infty)^2$ to maximize the next-step accuracy $\Psi(r,K,\eta,\sigma)$.
This step is the only place where non-compactness of the admissible set could, a priori, obstruct existence.

\begin{theorem}[Optimal Filtering and Invariance Principle]
\label{thm:invariance}
Fix $r\in(0,1)$ and $\eta\in[0,1)$. The one-step maximization problem
\begin{equation}
\label{eq:one_step_problem}
\sup_{K>0,\ \sigma>0}\ \Psi(r,K,\eta,\sigma)
\end{equation}
admits a \emph{unique} maximizer $(K^*(r),\sigma^*(r))\in(0,\infty)^2$. It is given by
\begin{equation}
\label{eq:optimal_controls}
K^*(r)=\frac{1-r^2}{2(1-r^2)+\beta^2},\qquad \sigma^*(r)=r.
\end{equation}
Moreover, under this optimal filtering choice, the induced transition becomes \textbf{independent} of
$\eta$ and reduces to the one-dimensional map
\begin{equation}
\label{eq:invariant_map_32}
\Psi\bigl(r,K^*(r),\eta,\sigma^*(r)\bigr)=:\Phi(r)
=\lambda\sqrt{\,r^2+\frac{(1-r^2)^2}{\beta^2+2(1-r^2)}\,}.
\end{equation}
\end{theorem}

\begin{proof}[Proof (main steps; full derivations in Appendices~\ref{app:variance_derivation} and \ref{app:invariance_proof})]
Fix $r\in(0,1)$ and $\eta\in[0,1)$. We show that the one-step problem
\eqref{eq:one_step_problem} attains its supremum on the non-compact domain $(0,\infty)^2$,
and that the maximizer is unique and interior.

\medskip
\noindent\textbf{Step 1 (Explicit representation and signal-matching).}
Starting from the explicit transition formula $\Psi(r,K,\eta,\sigma)$ in
Theorem~\ref{thm:transition_function}, write
\[
\Psi(r,K,\eta,\sigma)
=\lambda\,\frac{\mathcal{N}(K,\sigma)}{\sqrt{\Lambda^2(K,\sigma)}},
\]
where $\mathcal{N}$ is the covariance numerator and $\Lambda^2$ is the pre-scaling variance.
The algebraic expressions of $\mathcal{N}$ and $\Lambda^2$ are derived in
Appendix~\ref{app:variance_derivation} and reorganized into a ratio of quadratic forms in
Appendix~\ref{app:invariance_proof}.

Appendix~\ref{app:invariance_proof} further shows that, after optimizing over $K$ for each fixed
$\sigma$, the resulting upper envelope depends on $\sigma$ only through $(\sigma-r)^2$ and is
\emph{strictly decreasing} in $(\sigma-r)^2$. Hence the unique maximizer in the $\sigma$-direction
is
\[
\sigma^*(r)=r,
\]
which is the canonical \emph{signal-matching} scaling.

\medskip
\noindent\textbf{Step 2 (Optimal gain and $\eta$-elimination).}
Under $\sigma=r$, Appendix~\ref{app:invariance_proof} shows that both the numerator and the variance
simplify to
\[
\mathcal{N}(K,r)=r^2+K(1-r^2),
\qquad
\Lambda^2(K,r)=r^2+K^2\bigl(\beta^2+2(1-r^2)\bigr),
\]
which are independent of $\eta$. Therefore, for $\sigma=r$ the transition is independent of $\eta$:
\[
\Psi(r,K,\eta,r)=\lambda\,\frac{r^2+K(1-r^2)}
{\sqrt{r^2+K^2\bigl(\beta^2+2(1-r^2)\bigr)}}.
\]
The right-hand side is differentiable on $(0,\infty)$ and has a unique critical point.
Appendix~\ref{app:invariance_proof} verifies that this critical point is the unique global maximizer on
the non-compact domain $K>0$, and it is given by
\[
K^*(r)=\frac{1-r^2}{2(1-r^2)+\beta^2}>0.
\]
This proves attainment, uniqueness, and interiority of the maximizer $(K^*(r),\sigma^*(r))$.

\medskip
\noindent\textbf{Step 3 (Invariant map).}
Substituting $(K^*(r),\sigma^*(r))$ into $\Psi$ yields
\[
\Psi\bigl(r,K^*(r),\eta,\sigma^*(r)\bigr)
=\lambda\sqrt{\,r^2+\frac{(1-r^2)^2}{\beta^2+2(1-r^2)}\,}
=:\Phi(r),
\]
which is independent of $\eta$ and equals \eqref{eq:invariant_map_32}. This completes the proof.
\end{proof}

\begin{remark}[Interpretation]
\label{rem:invariance_interpretation}
The scaling $\sigma^*(r)=r$ matches the rating dispersion to the current estimation accuracy, preventing
noise amplification in the update step. Under this optimal scaling, the interaction topology $\eta$ affects
\emph{who plays whom} (and thus instantaneous utility) but does not affect \emph{how much information} each
match generates about skills.
\end{remark}

\subsection{Reduced Dynamics and Monotone Value of Accuracy}
\label{subsec:reduced_dynamics}
Under the optimal filtering rule of Theorem~\ref{thm:invariance}, the accuracy dynamics reduce to the
scalar recursion
\[
r_{t+1}=\Phi(r_t),
\]
which is independent of the platform's sorting choice $\eta_t$. This reduction is the key input for the
separation result in the next subsection: it implies that, once the learning engine is optimally calibrated,
choices of $\eta_t$ affect only current-period utility and not future accuracy.

\begin{proposition}[Analysis of the Kinetic Map]
\label{prop:equilibrium}
Let $\lambda \in (0, 1)$ and $\beta^2 > 0$. The map $\Phi: [0, 1] \to (0, \lambda]$ defined in
\eqref{eq:invariant_map_32} satisfies the following properties:
\begin{enumerate}
    \item \emph{Invariance and Monotonicity}: $\Phi$ maps the interval $[0, \lambda]$ into itself and is strictly increasing on $[0, 1]$.
    \item \emph{Existence and Uniqueness}: There exists a unique fixed point $r_\infty \in (0, \lambda)$ satisfying $r_\infty = \Phi(r_\infty)$.
    \item \emph{Global Stability}: For any initial condition $r_0 \in [0, 1)$, the sequence $r_{t+1} = \Phi(r_t)$ converges monotonically to $r_\infty$.
\end{enumerate}
\end{proposition}

\begin{proof}
We analyze the properties of the function $\Phi(r)$.

\noindent \textit{1. Boundedness and Monotonicity.}
Recall the definition:
\begin{equation*}
    \Phi(r) = \lambda \sqrt{r^2 + \frac{(1 - r^2)^2}{\beta^2 + 2(1 - r^2)}}.
\end{equation*}
At the boundaries, we have $\Phi(0) = \lambda / \sqrt{\beta^2 + 2} > 0$ and $\Phi(1) = \lambda$.
Since the term inside the square root is strictly bounded by 1 for all $r \in [0, 1]$, it holds that $\Phi(r) < \lambda$ for all $r < 1$. Thus, $\Phi$ maps $[0, \lambda]$ into $(0, \lambda]$.
To show monotonicity, let $x = r^2$ and consider the function $g(x) = x + \frac{(1-x)^2}{\beta^2 + 2(1-x)}$ for $x \in [0, 1]$. Differentiating with respect to $x$, we obtain
\begin{equation*}
    g'(x) = 1 - \frac{2(1-x)(\beta^2 + 2(1-x)) + 2(1-x)^2}{(\beta^2 + 2(1-x))^2} = 1 - \frac{2(1-x)[\beta^2 + 3(1-x)]}{(\beta^2 + 2(1-x))^2}.
\end{equation*}
A straightforward algebraic check confirms that $g'(x) > 0$ for $\beta^2 > 0$ and $x \in [0, 1)$. Since $\Phi(r) = \lambda \sqrt{g(r^2)}$, the map $\Phi$ is strictly increasing in $r$.

\noindent \textit{2. Uniqueness of the Fixed Point.}
The fixed point condition $r = \Phi(r)$ is equivalent to $r^2 = \lambda^2 g(r^2)$. Letting $x = r^2$, we seek a solution to $x = \lambda^2 g(x)$ in the interval $(0, \lambda^2)$.
This equation can be rearranged into a polynomial equation. Substituting the expression for $g(x)$:
\begin{equation*}
    x = \lambda^2 \left( x + \frac{(1-x)^2}{\beta^2 + 2 - 2x} \right).
\end{equation*}
Rearranging terms yields:
\begin{equation*}
    x (1 - \lambda^2) [\beta^2 + 2 - 2x] = \lambda^2 (1-x)^2.
\end{equation*}
Define $H(x) = x (1 - \lambda^2) [\beta^2 + 2 - 2x] - \lambda^2 (1-x)^2$.
We observe that $H(x)$ is a quadratic function opening downwards (the coefficient of $x^2$ is $-2(1-\lambda^2) - \lambda^2 = \lambda^2 - 2 < 0$).
Checking the boundary values:
\begin{itemize}
    \item $H(0) = -\lambda^2 < 0$.
    \item $H(\lambda^2) = \lambda^2 (1 - \lambda^2) [\beta^2 + 2 - 2\lambda^2] - \lambda^2 (1-\lambda^2)^2 = \lambda^2 (1-\lambda^2) [\beta^2 + 1 - \lambda^2] > 0$.
\end{itemize}
Since $H(0) < 0$ and $H(\lambda^2) > 0$, and $H$ is continuous, there exists at least one root in $(0, \lambda^2)$. Furthermore, since $H(1) = (1-\lambda^2)\beta^2 > 0$ and the parabola opens downward, the other root must be greater than 1 (or negative, but $H(0)<0$ implies the positive root is unique in the feasible range). Thus, there is exactly one solution $x_\infty \in (0, \lambda^2)$, which implies a unique fixed point $r_\infty = \sqrt{x_\infty} \in (0, \lambda)$.

\noindent \textit{3. Convergence.}
Since $\Phi$ is strictly increasing, continuous, and maps $[0, \lambda]$ into itself with a unique fixed point $r_\infty$, standard results on one-dimensional maps guarantee global stability. Specifically, for any $r_0 \in [0, r_\infty)$, the sequence is strictly increasing and bounded above by $r_\infty$, while for $r_0 \in (r_\infty, 1]$, it is strictly decreasing and bounded below. In both cases, the sequence converges to $r_\infty$.
\end{proof}

\begin{remark}[The Red Queen Effect]
\label{rem:red_queen}
The upper bound $r_\infty<\lambda$ quantifies an entropic ceiling created by skill diffusion:
even under optimal filtering, rating accuracy cannot exceed $\lambda$ because skills continuously
evolve over time. The system must therefore keep updating ratings not to reach perfect information,
but to offset the ongoing dissipation of information.
\end{remark}

Next define the instantaneous net utility and its envelope:
\[
U(r,\eta):=\eta r^2 - C(\eta),\qquad
V(r):=\max_{\eta\in[0,1)} U(r,\eta).
\]

\begin{lemma}[Accuracy has monotone value]
\label{lem:V_increasing}
Assume $C:[0,1)\to\mathbb{R}_+$ is convex, nondecreasing, satisfies $C(0)=0$, and $\lim_{\eta\uparrow 1}C(\eta)=+\infty$.
Then for each $r\in[0,1]$ the maximization problem defining $V(r)$ admits at least one solution
$\eta^*(r)\in[0,1)$, and the envelope $V(r)$ is \emph{nondecreasing} in $r$ on $(0,1]$.
\end{lemma}

\begin{proof}
Fix $\eta\in[0,1)$. The function $r\mapsto U(r,\eta)=\eta r^2-C(\eta)$ is strictly increasing on $(0,1]$
whenever $\eta>0$. Since $C(\eta)\to+\infty$ as $\eta\uparrow 1$, the maximizer cannot be at $\eta=1$,
and convexity ensures upper semicontinuity of the objective on the compact set $[0,1-\varepsilon]$ for
any $\varepsilon>0$, implying existence of a maximizer $\eta^*(r)\in[0,1)$.
To show strict monotonicity of $V$, take $0<r_1<r_2\le1$. Let $\eta_1\in\arg\max_\eta U(r_1,\eta)$.
Then
\[
V(r_2)\ \ge\ U(r_2,\eta_1)\ =\ U(r_1,\eta_1) + \eta_1(r_2^2-r_1^2)\ =\ V(r_1)+\eta_1(r_2^2-r_1^2) \ge\ V(r_1).
\]
\end{proof}

\begin{remark}[Connection to separation]
\label{rem:why_needed_for_separation}
Lemma~\ref{lem:V_increasing} implies that higher accuracy $r$ strictly increases the best attainable
instantaneous net utility $V(r)$. Together with Proposition~\ref{prop:equilibrium} (which shows that
$\Phi$ is strictly increasing), this monotonicity will allow us to conclude that maximizing $r_{t+1}$
at each step maximizes the entire continuation value, which is the key step behind the separation
result in Corollary~\ref{cor:separation}.
\end{remark}

\subsection{Optimal Control Strategy}
Thanks to the Invariance Principle (Theorem \ref{thm:invariance}), the infinite-horizon optimization problem \eqref{eq:mfc_problem} splits into two tractable static problems.

\begin{corollary}[Optimal separated control policy]
\label{cor:separation}
The following control policy is optimal. For each $t$, given the current state $r_t$:

\begin{enumerate}
    \item \emph{Filtering (Greedy):} Choose $(K_t,\sigma_t)$ to maximize the next-state correlation:
    \begin{equation}
        (K_t^*,\sigma_t^*)
        \in
        \operatorname*{argmax}_{K>0,\,\sigma> 0}\; \Psi(r_t, K_t, \eta_t, \sigma_t),
    \end{equation}
    which is implemented by the Kalman gain and signal-matched scaling in \eqref{eq:optimal_controls}.
    In particular, any maximizer $(K_t^*,\sigma_t^*)$ is optimal and this choice is independent of the cost function $C(\cdot)$ and $\eta_t$.

    \item \emph{Matching (Static):} Choose $\eta_t$ myopically as a maximizer of instantaneous net utility:
    \begin{equation}
        \eta_t^*
        \in
        \operatorname*{argmax}_{\eta \in [0,1)} \left\{ \eta r_t^2 - C(\eta) \right\}.
    \end{equation}
\end{enumerate}
\end{corollary}

\begin{proof}
Let $\pi=\{(K_t,\sigma_t,\eta_t)\}_{t\ge 0}$ be any admissible policy and let
$\{r_t\}_{t\ge 0}$ be the induced state sequence under the transition
$r_{t+1}=\Psi(r_t,K_t,\eta_t,\sigma_t)$.
Define the instantaneous net utility and its envelope
\[
U(r,\eta):=\eta r^2-C(\eta),
\qquad
V(r):=\max_{\eta\in[0,1)}U(r,\eta).
\]
Then $U(r_t,\eta_t)\le V(r_t)$ for every $t$.

Next, by Theorem~\ref{thm:invariance}, for every current state $r\in(0,1)$ and any $\eta\in[0,1)$,
\[
\sup_{K>0,\ \sigma>0}\Psi(r,K,\eta,\sigma)=:\Phi(r),
\]
and the supremum is attained at the maximizer $(K^*(r),\sigma^*(r))$ given in \eqref{eq:optimal_controls}.
In particular, the maximizer $(K^*(r),\sigma^*(r))$ depends only on $r$ and is independent of both $C(\cdot)$ and $\eta$.
Therefore, along any admissible policy $\pi$ we have the pointwise bound
\[
r_{t+1}=\Psi(r_t,K_t,\eta_t,\sigma_t)\le \Phi(r_t)\qquad\text{for all }t\ge 0.
\]

Now define the benchmark sequence $\{\bar r_t\}_{t\ge 0}$ by $\bar r_0=r_0$ and
\[
\bar r_{t+1}=\Phi(\bar r_t)\qquad\text{for all }t\ge 0.
\]
Since $\Phi$ is increasing (Proposition~\ref{prop:equilibrium}), the inequality $r_{t+1}\le \Phi(r_t)$ implies by induction that
$r_t\le \bar r_t$ for all $t\ge 0$. Hence, using that $V$ is nondecreasing in $r$ (Lemma~\ref{lem:V_increasing}, with the monotonicity statement),
\[
J(\pi)=\sum_{t=0}^\infty \delta^t U(r_t,\eta_t)
\le \sum_{t=0}^\infty \delta^t V(r_t)
\le \sum_{t=0}^\infty \delta^t V(\bar r_t).
\]

Finally, consider the policy stated in the corollary:
(i) choose $(K_t^*,\sigma_t^*)\in\arg\max_{K,\sigma}\Psi(r_t,K,\eta_t,\sigma)$,
which implements $(K^*(r_t),\sigma^*(r_t))$ from \eqref{eq:optimal_controls} and yields $r_{t+1}=\Phi(r_t)$;
(ii) choose $\eta_t^*\in\arg\max_{\eta\in[0,1)}\{\eta r_t^2-C(\eta)\}$, so that $U(r_t,\eta_t^*)=V(r_t)$.
Under this policy we have $r_t=\bar r_t$ for all $t$ and $U(r_t,\eta_t^*)=V(r_t)$ for all $t$, so the achieved objective equals
\[
\sum_{t=0}^\infty \delta^t V(\bar r_t),
\]
which matches the upper bound above. Therefore the stated separated policy is optimal.
\end{proof}

\begin{remark}[The Principle of Control Separation and its Distinction from LQG]
\label{rem:separation}
This result justifies a modular design for competitive platforms by confirming a strong form of the \emph{Separation Principle}. The overall complexity of the Mean-Field Control problem is dramatically reduced as the two core functions can be treated independently:

\begin{enumerate}
    \item The \emph{Estimation Engine} (Filtering): The control pair $(K_t, \sigma_t)$ is dedicated to maximizing the rate of information acquisition, defining the dynamic map $r_{t+1} = \Phi(r_t)$ (Theorem \ref{thm:invariance}).
    \item The \emph{Matchmaking Engine} (Utility Control): The parameter $\eta_t$ is determined solely by maximizing the instantaneous net utility $U(r_t, \eta_t) = \eta_t r_t^2 - C(\eta_t)$ (Corollary \ref{cor:separation}).
\end{enumerate}

This result differs fundamentally from the classical LQG Separation Principle (Certainty Equivalence). In the classical setting, the optimal control uses the estimated state ($\hat{x}_t$) but the filtering process is independent of the control strategy. In our system, the control variables directly influence the information flow via the variance term $\Lambda^2$. Our Invariance Principle (Theorem \ref{thm:invariance}) reveals a stronger structural separation of control variables: the optimal selection of the update parameters $(K_t^*, \sigma_t^*)$ eliminates the influence of the interaction topology $\eta_t$ on the learning dynamics.

\noindent \emph{Structural Simplification ($\eta$ Elimination):}
Crucially, the Separation Principle implies that the maximization over $\eta_t$ fully decouples from the maximization over the dynamic controls $(K_t, \sigma_t)$ in the Bellman equation. The problem reduces to finding $\eta_t^*$ locally based on the instantaneous cost-benefit trade-off: $\max_{\eta} \{\eta r_t^2 - C(\eta)\}$. This structural simplification effectively eliminates the complexity associated with solving the Dynamic Programming problem over the multi-dimensional control space, making the full optimal control problem analytically tractable.
\end{remark}

\subsection{Thermodynamic Interpretation of the Matching Policy}
\label{sec:thermodynamics}

The optimization problem for the matchmaking intensity $\eta_t$, as derived in Corollary \ref{cor:separation}, takes the form of a static maximization of instantaneous welfare. Let $u_t := r_t^2$ denote the "information potential" (or squared correlation). The problem is written as:
\begin{equation}
    V(u_t) = \sup_{\eta \in [0, 1)} \left\{ \eta u_t - C(\eta) \right\}.
\end{equation}
From the perspective of convex analysis, the value function $V(\cdot)$ is identified as the \emph{Legendre-Fenchel transform} (or convex conjugate) of the cost function $C(\eta)$. This structural observation invites a thermodynamic interpretation of the system's control.
Here, $u_t$ acts as an external field driving the system towards stricter sorting, while the matchmaking intensity $\eta$ plays the role of the order parameter (analogous to magnetization in a spin system). The cost function $C(\eta)$ represents the entropic resistance or the energetic barrier to maintaining order.

\paragraph{Queueing-Theoretic Cost and Critical Phase Transition.}
To derive explicit analytical insights, we introduce a barrier potential cost function. This models the friction of organizing a queue, where waiting times diverge as matching requirements become infinitely strict ($\eta \to 1$). We adopt the form:
\begin{equation}
\label{eq:barrier_cost}
    C(\eta) = \kappa_C \left( \frac{1}{1-\eta} - 1 \right), \quad \kappa_C > 0,
\end{equation}
where $\kappa_C$ represents the characteristic cost parameter (e.g., the ratio of opportunity cost of delay to match utility).
Solving the optimization problem with this specific cost yields a closed-form control law exhibiting critical behavior.

\begin{proposition}[Matchmaking Phase Transition]
\label{prop:phase_transition}
Consider the cost function \eqref{eq:barrier_cost}. The optimal matchmaking policy $\eta^*(r_t)$ is given by:
\begin{equation}
\label{eq:phase_transition}
    \eta^*(r_t) = \max \left( 0, \, 1 - \frac{\sqrt{\kappa_C}}{r_t} \right).
\end{equation}
This policy implies a continuous phase transition at the critical accuracy threshold $r_c = \sqrt{\kappa_C}$.
\begin{itemize}
    \item \emph{Disordered Phase ($r_t \le r_c$):} If the rating information is insufficient ($r_t \le \sqrt{\kappa_C}$), the optimal policy is zero intervention ($\eta^* = 0$). The system operates in a random matching regime to prioritize throughput over quality.
    \item \emph{Ordered Phase ($r_t > r_c$):} Once accuracy exceeds the critical threshold, the optimal intensity becomes positive and strictly increases with $r_t$, asymptotically approaching strict matching ($\eta \to 1$).
\end{itemize}
\end{proposition}

\begin{proof}
Let $J(\eta) = \eta r_t^2 - C(\eta)$. Since $C(\eta)$ is strictly convex on $[0, 1)$, $J(\eta)$ is strictly concave. We look for the stationary point by solving the first-order condition $J'(\eta) = 0$:
\begin{equation*}
    r_t^2 - C'(\eta) = r_t^2 - \frac{\kappa_C}{(1-\eta)^2} = 0.
\end{equation*}
Solving for $\eta$, we obtain the unconstrained root:
\begin{equation*}
    (1-\eta)^2 = \frac{\kappa_C}{r_t^2} \implies \eta = 1 - \frac{\sqrt{\kappa_C}}{r_t}.
\end{equation*}
We must satisfy the constraint $\eta \in [0, 1)$.
Since $C(\eta) \to \infty$ as $\eta \to 1$, the upper bound is naturally satisfied.
For the lower bound, if $r_t \le \sqrt{\kappa_C}$, then $1 - \sqrt{\kappa_C}/r_t \le 0$. Since $J(\eta)$ is concave and decreasing for $\eta$ beyond the peak, the maximum on the interval $[0, 1)$ occurs at the boundary $\eta = 0$.
Conversely, if $r_t > \sqrt{\kappa_C}$, the root lies in $(0, 1)$. Combining these cases yields the expression \eqref{eq:phase_transition}.
\end{proof}

This result provides a theoretical foundation for the "calibration phase" widely implemented in competitive platforms. It suggests that loose matchmaking (or random matching) is not merely a heuristic for reducing wait times, but the \emph{welfare-maximizing strategy} when the platform lacks sufficient knowledge ($r_t < \sqrt{\kappa_C}$) to justify the cost of sorting. Mathematically, the system exhibits a second-order phase transition, where the order parameter $\eta^*$ leaves zero continuously but with a discontinuous derivative at $r_c$.

\begin{figure}[htbp]
    \centering
    \includegraphics[width=\textwidth]{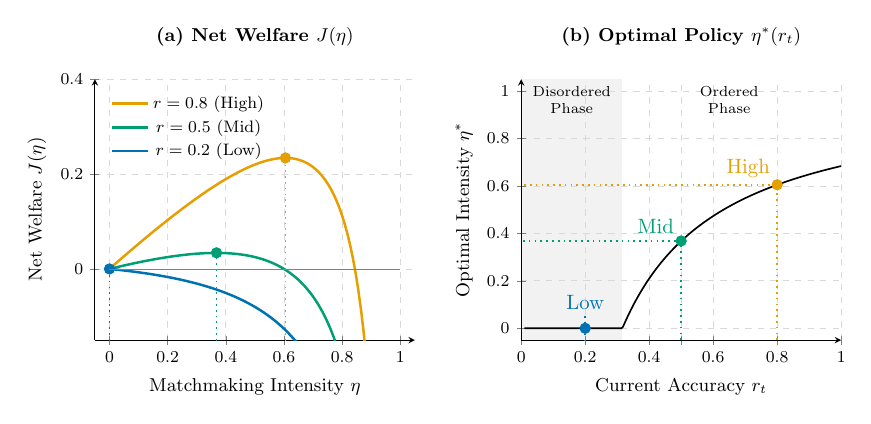}
    \caption{\emph{Matchmaking Accuracy and Optimal Intervention Policy.}
    (a) Net welfare function $J(\eta)$ for varying accuracy levels. Higher accuracy ($r$) increases the marginal benefit of matching, shifting the optimal intensity peak to the right.
    (b) Optimal matchmaking intensity $\eta^*$ as a function of current accuracy $r_t$. The system exhibits a continuous phase transition: below the critical threshold $r_c$ (Disordered Phase), the optimal policy is zero intervention, whereas in the Ordered Phase, the intensity scales with accuracy.}
    \label{fig:optimal_policy}
\end{figure}

\paragraph{Optional continuous-time perspective.}
For readers interested in a diffusion approximation, Appendix~\ref{app:ct_limit} derives a McKean--Vlasov SDE and the associated Fokker--Planck equation.
This continuous-time view is not used in the control results in Section~\ref{sec:analysis} nor in the numerical experiments in Section~\ref{sec:numerics}.

\section{Numerical Validation}
\label{sec:numerics}

In this section, we present numerical investigations to validate the kinetic theory and the control strategies derived in the previous sections. We employ a Monte Carlo particle method to simulate the $N$-particle stochastic system governed by \eqref{eq:micro_skill}--\eqref{eq:micro_update_pre} and compare the empirical statistics with the predictions of the macroscopic mean-field model. Our simulations focus on three aspects: (i) the quantitative rate of convergence of the empirical measure to the mean-field limit, (ii) the asymptotic behavior of rating accuracy in the presence of skill drift, and (iii) the verification of the Invariance Principle (Theorem \ref{thm:invariance}).

\subsection{Convergence to the Mean-Field Limit}
First, we quantify the error between the finite-$N$ particle dynamics and the deterministic map $r_{t+1} = \Psi(r_t, \dots)$ derived in Theorem \ref{thm:transition_function}. This step is crucial to justify the use of the reduced-order model for control design.

\paragraph{Simulation Setup.}
We simulate systems with population sizes varying logarithmically from $N = 10^1$ to $10^5$. The skill persistence is set to $\lambda = 0.99$, and the observation noise variance is $\beta^2 = 1.0$. The control parameters are fixed at a constant gain $K=0.1$, random matching $\eta=0.0$, and unit scaling $\sigma=1.0$.

Because the theoretical normalization $X_{i,0}\equiv 0$ is degenerate (so $\sigma_0=0$ and the correlation-based accuracy is not defined at $t=0$), we start reporting simulation moments from the first non-degenerate period and relabel that time as $t=0$ for convenience.
To isolate dynamical fluctuations from initial sampling noise, we employ a variance reduction technique at this reported initial time: the finite population vectors $\boldsymbol{\rho}_0$ and $\boldsymbol{X}_0$ are normalized and orthogonalized so that the initial empirical moments exactly match the theoretical moments (unit variances and zero correlation, $r_0=0$).
The simulation then proceeds by sequentially applying the interaction update and the skill drift step, consistent with the discrete-time definition.

\paragraph{Results and Scaling Law.}
Figure \ref{fig:convergence} (Left) displays the time evolution of the empirical rating accuracy $r_t^{(N)}$. The stochastic trajectories of the particle system fluctuate around the theoretical mean-field curve $r_t^{(\infty)}$. While finite-size effects are noticeable for small $N$, the trajectories converge rapidly to the deterministic limit as the population size increases.

To quantify the convergence rate, we compute the time-averaged $L^2$ error, defined as $E_N = (\frac{1}{T} \sum_{t=1}^T |r_t^{(N)} - r_t^{(\infty)}|^2)^{1/2}$.
Figure \ref{fig:convergence} (Right) plots $E_N$ against the population size $N$ on a log-log scale.
Numerical regression in the asymptotic range $N \in [10^3, 10^5]$ yields a slope of approximately $-0.51$. This provides strong empirical evidence that the error scales as $N^{-1/2}$, consistent with the Central Limit Theorem. Specifically, the mean squared error follows the scaling law:
\begin{equation}
    \label{eq:mse_scaling}
    \mathbb{E} \left[ | r_t^{(N)} - r_t^{(\infty)} |^2 \right] \propto \frac{1}{N}.
\end{equation}
This result numerically validates the Propagation of Chaos property and justifies the Gaussian closure approximation for analyzing systems with realistic population sizes.

\begin{figure}[htbp]
    \centering
    \includegraphics[width=1.0\textwidth]{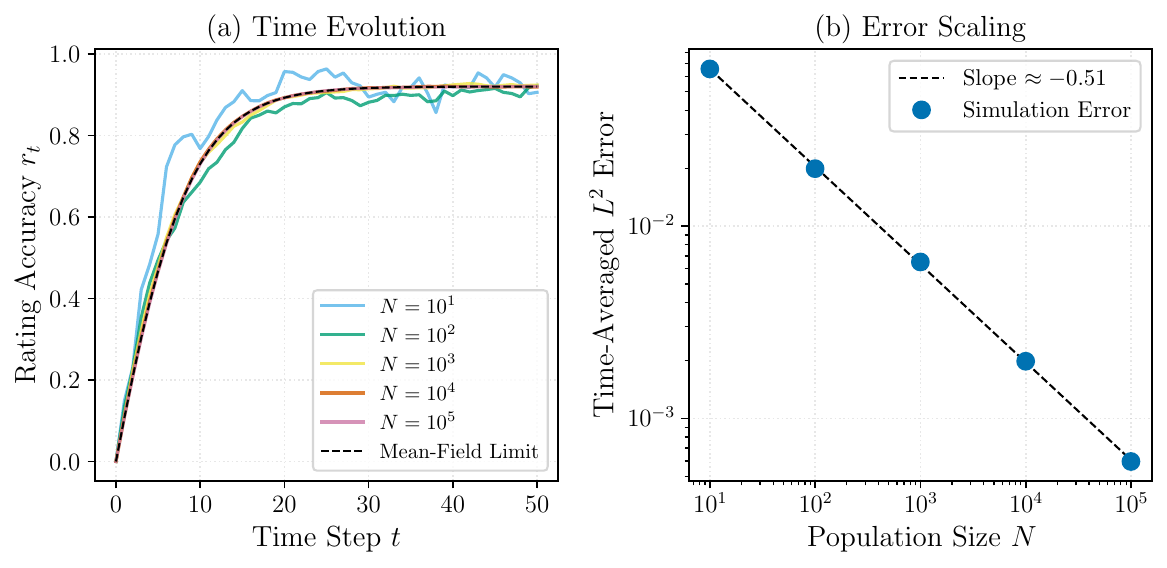}
    \caption{\emph{Convergence to the Mean-Field Limit.} (Left) Time evolution of rating accuracy for varying population sizes $N$. The dashed line represents the theoretical mean-field limit. (Right) Scaling of the time-averaged $L^2$ error with respect to $N$. The error follows the $O(N^{-1/2})$ law characteristic of mean-field approximations, as indicated by the best-fit slope of $-0.51$.}
    \label{fig:convergence}
\end{figure}

\subsection{Entropic Limits and the Red Queen Effect}
Next, we quantitatively analyze the stationary state of the kinetic map to assess the impact of skill non-stationarity. We solve for the fixed point $r_\infty = \Phi(r_\infty)$ assuming the platform employs the optimal Kalman gain $K^*$ derived in Section \ref{sec:analysis}.

Figure \ref{fig:drift_impact} plots the equilibrium accuracy $r_\infty$ as a function of the observation noise variance $\beta^2$ for distinct skill persistence parameters $\lambda \in \{1.0, 0.995, 0.99, 0.95\}$. The results highlight a fundamental dichotomy between static and dynamic environments:

\begin{itemize}
    \item \emph{Conservative Limit ($\lambda=1.0$):} In the static case, the accuracy converges asymptotically to $1.0$ for any finite observation noise. This recovers the classical consistency result of rating systems: with infinite time, the true parameters are perfectly identifiable.
    \item \emph{Dissipative Regime ($\lambda < 1.0$):} In the presence of skill drift, the equilibrium accuracy is strictly bounded by the persistence parameter, i.e., $r_\infty < \lambda$, confirming the theoretical bound in Proposition \ref{prop:equilibrium}. Notably, the system exhibits high sensitivity to the drift parameter. Even a marginal deviation from stationarity (e.g., $\lambda=0.99$) induces a significant drop in achievable accuracy, especially under high noise conditions ($\beta^2 > 1$).
\end{itemize}

This behavior numerically validates the "Red Queen" regime discussed in Remark~\ref{rem:red_queen}. The gap between the curves for $\lambda=1.0$ and $\lambda=0.99$ illustrates that the continuous dissipation of information (entropic drift) imposes a hard ceiling on estimation fidelity, which cannot be overcome by simply accumulating more data.

\begin{figure}[htbp]
    \centering
    \includegraphics[width=0.7\textwidth]{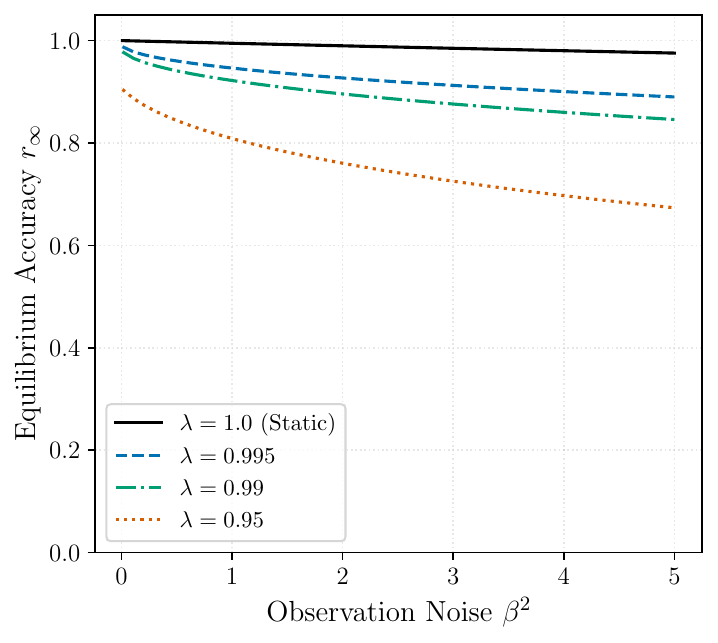}
    \caption{\emph{The Red Queen Effect.} Equilibrium accuracy $r_\infty$ as a function of observation noise $\beta^2$. The dashed line ($\lambda=1.0$) represents the ideal static limit where perfect learning is possible. Solid lines ($\lambda < 1.0$) show the dissipative regime where accuracy is strictly bounded, degrading sharply as the signal-to-noise ratio decreases.}
    \label{fig:drift_impact}
\end{figure}

\subsection{Validation of the Invariance Principle}
Finally, we provide numerical evidence for the Invariance Principle (Theorem \ref{thm:invariance}), which asserts that under optimal scaling, the learning dynamics become independent of the interaction kernel structure. We compare the evolution of rating accuracy starting from an initial state $r_0=0.1$ under two distinct control regimes. 
Here $r_0$ refers to the reported initial time in the simulations, which corresponds to the first non-degenerate step after initialization (cf.~the simulation setup).
We compare:
\begin{enumerate}
    \item \emph{Suboptimal Regime (Fixed Scale):} The rating scale is rigidly fixed at $\sigma_t \equiv 1$, ignoring the current estimation confidence.
    \item \emph{Optimal Regime (Adaptive Scale):} The rating scale is adapted dynamically as $\sigma_t = r_t$, implementing the signal-matched scaling derived in \eqref{eq:optimal_controls}.
\end{enumerate}

In both regimes, we simulate the system under three distinct interaction intensities: Random Matching ($\eta=0$), Intermediate ($\eta=0.5$), and Strict Matching ($\eta=0.9$). Figure \ref{fig:decoupling} presents the results. In the Suboptimal Regime (dashed lines), the trajectories diverge significantly based on $\eta$. Higher matchmaking intensity ($\eta=0.9$) leads to faster convergence. This dependency implies that when the rating scale is not calibrated to the signal quality, the matchmaking mechanism must compensate to control the effective noise variance. In stark contrast, in the Optimal Regime (solid lines), the trajectories for $\eta=0, 0.5, 0.9$ \emph{collapse onto a single universal curve}. This data collapse visually confirms the algebraic cancellation shown in the proof of Theorem \ref{thm:invariance}: the information content of the update is invariant to the assortment $\eta$, provided the macroscopic scale $\sigma$ acts as a proper normalization factor. This result corroborates the Separation Principle (Corollary \ref{cor:separation}), demonstrating that the filtering problem (maximizing $r_t$) can be solved completely independently of the matchmaking policy. \begin{figure}[htbp] \centering \includegraphics[width=0.7\textwidth]{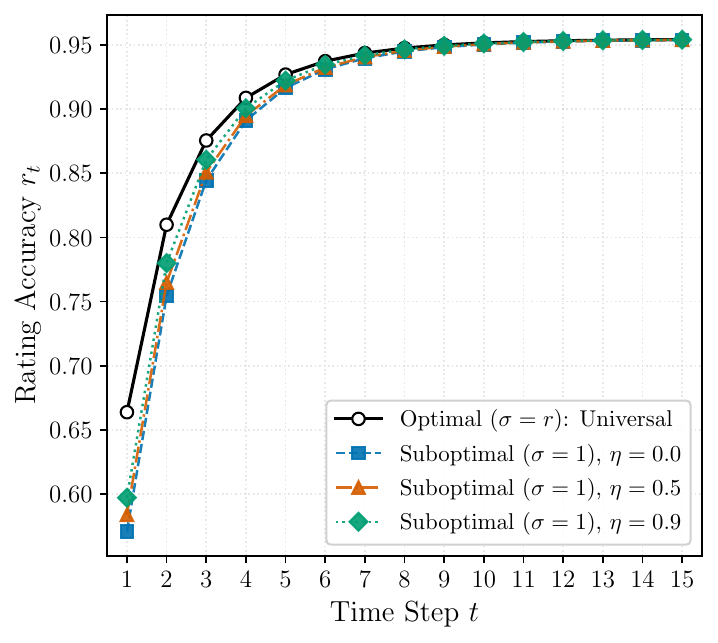} \caption{\emph{Verification of the Invariance Principle.} Dashed lines show trajectories under fixed rating scaling ($\sigma \equiv 1$), where the learning speed depends on the matchmaking intensity $\eta$. Solid lines show trajectories under optimal scaling ($\sigma = r_t$), where all curves collapse onto a single universal trajectory regardless of $\eta$. This demonstrates that under optimal control, the learning dynamics are effectively decoupled from the matchmaking topology.} \label{fig:decoupling} \end{figure}

\section{Conclusion}
\label{sec:conclusion}

We developed a kinetic/mean-field framework for platform design in rating and matchmaking systems with drifting skills.
Starting from a microscopic interacting particle model, we derived a McKean--Vlasov description of the joint skill--rating dynamics and proved an exact Gaussian closure in the Gaussian regime.
This yields a finite-dimensional, analytically tractable state representation for rating accuracy and enables an explicit solution of the associated mean-field control problem when the platform controls
\[
u_t=(K_t,\eta_t,\sigma_t),
\]
where $\sigma_t$ is interpreted operationally as the rating scale applied in period $t$.

Three structural lessons emerge.

\begin{enumerate}
    \item \emph{Entropic limits under skill drift (Red Queen effect).}
    Even under optimal filtering, accuracy cannot converge to one in a non-stationary environment.
    With observation noise $\beta^2>0$ and persistence $\lambda\in(0,1)$, the induced accuracy map admits a unique globally stable fixed point $r_\infty\in(0,\lambda)$, providing a sharp ceiling on long-run learning.

    \item \emph{Invariance of information flow under signal-matched scaling.}
    Proper scale control is decisive.
    When the platform applies signal-matched scaling (implemented as $\sigma_t=r_t$), the one-step accuracy gain becomes independent of assortativity $\eta_t$.
    Thus, assortative matching is not intrinsically required for fast learning: once ratings are correctly normalized, even high-entropy (random) interactions can be information-efficient.

    \item \emph{Separation of filtering and matching, and a welfare phase transition.}
    The invariance yields a separation of platform decisions.
    Filtering controls $(K_t,\sigma_t)$ are determined myopically by maximizing next-period accuracy, while matching $\eta_t$ solves a static instantaneous trade-off between match utility and sorting/latency costs.
    Under a queueing-type barrier cost, the optimal policy exhibits a threshold behavior: when current accuracy is low, random matching is optimal; once accuracy exceeds a critical level, increasingly strict sorting becomes welfare-improving.
\end{enumerate}

Methodologically, the paper illustrates how kinetic-limit tools can convert a high-dimensional interacting-agent algorithm into a low-dimensional control problem with transparent structure.
Substantively, the results point to a modular platform architecture: a properly calibrated \emph{estimation engine} (via scale control) can be designed largely independently of the \emph{matchmaking engine}, which can then be tuned to welfare and latency considerations.

Several extensions are natural.
First, moving beyond the Gaussian regime (e.g., heavy-tailed skills or nonlinear outcome models) would require new closure or approximation schemes.
Second, incorporating strategic behavior (e.g., rating manipulation or effort choice) would lead to a mean-field game formulation and raise incentive-design questions.
Finally, while our main analysis is discrete-time and self-contained, the Appendix provides an optional diffusion approximation (Appendix~\ref{app:ct_limit}) that may be useful for readers who prefer a continuous-time perspective; it is not used in the main control results.

\appendix

\section{Derivation of the Mean-Field Limit and Gaussian Closure}
\label{app:meanfield}

In this appendix, we provide the mathematical justification for the macroscopic dynamics derived in Section \ref{sec:model}. We explicitly formulate the mean-field limit as a McKean-Vlasov type process, justify the Gaussian closure property, and provide the detailed algebraic derivation of the variance dynamics (Theorem \ref{thm:transition_function}).

\subsection{The McKean-Vlasov Process Construction}
Let $(\bar{\rho}_t, \bar{X}_t)$ denote the state of a typical particle in the mean-field limit. Unlike the $N$-particle system where agents interact with specific peers, the typical particle interacts with a "shadow opponent" drawn from the conditional law of the population.

The skill dynamics are governed by the independent AR(1) process:
\begin{equation}
    \bar{\rho}_{t+1} = \lambda \bar{\rho}_t + \sqrt{1-\lambda^2} \xi_t, \quad \xi_t \sim \mathcal{N}(0,1).
\end{equation}

For the rating update step, we construct the shadow opponent's state $(\bar{\rho}'_t, \bar{X}'_t)$ to satisfy the statistical constraints imposed by the interaction kernel $\mathcal{K}_\eta$.
Given the agent's rating $\bar{X}_t$, the opponent's rating $\bar{X}'_t$ is generated via the coupling:
\begin{equation}
    \bar{X}'_t = \eta_t \bar{X}_t + \sqrt{1-\eta_t^2} Z_t,
\end{equation}
where $Z_t$ is an independent copy of $\bar{X}_t$ (i.e., $Z_t \sim \mathcal{N}(0, \sigma_t^2)$ and $Z_t \perp \bar{X}_t$). This construction ensures the required correlation $\Corr(\bar{X}_t, \bar{X}'_t) = \eta_t$.

Crucially, the opponent's skill $\bar{\rho}'_t$ is not independent but is drawn from the conditional distribution $f_t(\rho | x = \bar{X}'_t)$. Under the Gaussian ansatz, this is realized by the linear projection:
\begin{equation}
    \bar{\rho}'_t = \frac{r_t}{\sigma_t} \bar{X}'_t + \sqrt{1-r_t^2} \zeta_t,
\end{equation}
where $\zeta_t \sim \mathcal{N}(0, 1)$ is independent noise. This ensures $\Cov(\bar{\rho}'_t, \bar{X}'_t) = r_t \sigma_t$.

The continuous-state update rule is then given by the discrete-time McKean-Vlasov equation:
\begin{equation}
\label{eq:mckean_vlasov}
    \bar{X}_{t+1} = L_{t+1} \left[ \bar{X}_t + K_t \left( (\bar{\rho}_t - \bar{\rho}'_t) + \omega_t - (\bar{X}_t - \bar{X}'_t) \right) \right],
\end{equation}
where $\omega_t \sim \mathcal{N}(0, \beta^2)$ is the observation noise.

\subsection{Continuous-Time Diffusion Approximation (Optional)}
\label{app:ct_limit}

This subsection provides an optional continuous-time diffusion approximation of the discrete-time mean-field dynamics.
This perspective is not used elsewhere in the paper.

While the discrete-time formulation in Section~\ref{sec:model} captures the operational mechanics, analytically characterizing the global distribution benefits from the tools of stochastic calculus. By taking the continuous-time limit of the microscopic update rules, we validate the structural assumptions made in Section~\ref{sec:model}, specifically proving that the Gaussian closure is exact in the linear regime.

\subsubsection{Diffusion-limit McKean--Vlasov System}
We consider the high-frequency limit where interaction events occur at a rate $1/\Delta t$.
\paragraph{Continuous-time scaling.}
To take the continuous-time limit, we interpret the discrete-time model as a $\Delta t$-time-step
approximation with physical time $s=n\Delta t$. Accordingly, the control parameters may depend on
$\Delta t$, and we write $K_{\Delta t}$ and $\beta_{\Delta t}^2$.
We impose the diffusive scaling
\begin{equation}
\label{eq:ct_scaling}
\frac{K_{\Delta t}}{\Delta t}\to \kappa
\qquad\text{and}\qquad
\frac{\beta_{\Delta t}^2}{\Delta t}\to \Sigma_M^2
\qquad\text{as }\Delta t\downarrow 0,
\end{equation}
equivalently $K_{\Delta t}=\kappa\,\Delta t+o(\Delta t)$ and $\beta_{\Delta t}^2=\Sigma_M^2\,\Delta t+o(\Delta t)$.

Let $(\rho_t, X_t)$ denote the state of a representative agent in the mean-field limit. The dynamics are governed by the following system of coupled stochastic differential equations (SDEs):

\begin{equation}
\label{eq:coupled_sde}
\begin{cases}
    d\rho_t = -\theta \rho_t \, dt + \sigma_\rho \, dB_t, \\
    dX_t = \kappa (\rho_t - X_t) \, dt + \Sigma(\mu_t) \, dW_t,
\end{cases}
\end{equation}
where $B_t$ and $W_t$ are independent Wiener processes.
Here, $\theta \approx 1-\lambda$ represents the skill mean-reversion rate, and $\kappa$ is the continuous-time learning rate.

Crucially, the diffusion coefficient $\Sigma(\mu_t)$ for the rating process is determined by the mean-field interaction and defined by:
\begin{equation}
\label{eq:sigma_fpe_full}
\Sigma(\mu_t)^2
= \kappa^2\Big(\Sigma_M^2+\Var\!\big(\rho'_t-X'_t \,\big|\, X_t\big)\Big),
\end{equation}

\begin{lemma}[State-independence of the diffusion coefficient under Gaussian coupling]
\label{lem:sigma_state_indep}
Assume the Gaussian regime of Proposition~\ref{prop:gaussianity} so that $(\rho_t,X_t)$ is bivariate Gaussian with moments $(r_t,\sigma_t)$.
Generate the opponent by the linear coupling
\[
X_t'=\eta_t X_t+\sqrt{1-\eta_t^2}\,Z_t,\qquad
Z_t\sim\mathcal N(0,\sigma_t^2),\quad Z_t\perp X_t,
\]
and
\[
\rho_t'=\frac{r_t}{\sigma_t}X_t' + \sqrt{1-r_t^2}\,\zeta_t,\qquad
\zeta_t\sim\mathcal N(0,1),\quad \zeta_t\perp (X_t,X_t',Z_t).
\]
Then $\Var(\rho_t'-X_t'\mid X_t)$ is constant (does not depend on $X_t$) and equals
\begin{equation}
\label{eq:cond_var_rhop_xp}
\Var(\rho_t'-X_t'\mid X_t)
=(1-r_t^2) + (1-\eta_t^2)\,(\,r_t-\sigma_t\,)^2.
\end{equation}
Consequently, $\Sigma(\mu_t)$ depends on $\mu_t$ only through the global moments $(r_t,\sigma_t)$ and is independent of the local state $X_t$.
\end{lemma}

\begin{proof}
Write $\rho_t'-X_t'=\left(\frac{r_t}{\sigma_t}-1\right)X_t' + \sqrt{1-r_t^2}\,\zeta_t$.
Conditional on $X_t$, the variable $X_t'$ is Gaussian with variance $\Var(X_t'\mid X_t)=(1-\eta_t^2)\sigma_t^2$, which is independent of $X_t$.
Independence of $\zeta_t$ yields
\[
\Var(\rho_t'-X_t'\mid X_t)
=\left(\frac{r_t}{\sigma_t}-1\right)^2(1-\eta_t^2)\sigma_t^2+(1-r_t^2),
\]
which simplifies to \eqref{eq:cond_var_rhop_xp}.
\end{proof}

By Lemma~\ref{lem:sigma_state_indep}, the diffusion coefficient in \eqref{eq:sigma_fpe_full} is \emph{state-independent} (additive noise), which is a key ingredient behind the exact Gaussian closure in the continuous-time limit.

\paragraph{Addressing Multiplicative Noise Concern.}
Equation \eqref{eq:sigma_fpe_full} defines $\Sigma(\mu_t)$ through the conditional variance
$\Var(\rho'_t-X'_t\mid X_t)$.
Lemma~\ref{lem:sigma_state_indep} shows that, under the Gaussian coupling construction,
this conditional variance is \emph{constant} and does not depend on the local state $X_t$.
Therefore, $\Sigma(\mu_t)$ depends on $\mu_t$ only through global moments (e.g., $(r_t,\sigma_t)$),
and the rating SDE in \eqref{eq:coupled_sde} has \emph{additive} (state-independent) noise.
This rules out multiplicative-noise effects and is the key structural reason why Gaussianity is preserved exactly
in the continuous-time limit.

\subsubsection{The Fokker-Planck Equation}
The time evolution of the joint probability density $f(t, \rho, x)$ is governed by the associated Fokker-Planck equation (or Kolmogorov forward equation):
\begin{equation}
\label{eq:fp}
    \frac{\partial f}{\partial t} =  \underbrace{\theta \frac{\partial}{\partial \rho} (\rho f) + \frac{\sigma_\rho^2}{2} \frac{\partial^2 f}{\partial \rho^2}}_{\text{Skill Diffusion}}  
    + \underbrace{\kappa \frac{\partial}{\partial x} ((x - \rho) f) + \frac{\Sigma(\mu_t)^2}{2} \frac{\partial^2 f}{\partial x^2}}_{\text{Rating Update}}.
\end{equation}
This Partial Differential Equation (PDE) encapsulates the competition between the intrinsic entropy of skill drift (first line) and the information gain from rating updates (second line).

\subsubsection{Theoretical Justification for Gaussian Closure}
A fundamental property of SDE systems is that linearity of the drift combined with state-independent (additive) Gaussian noise ensures the preservation of Gaussianity.
Since the drift terms in the McKean-Vlasov system \eqref{eq:coupled_sde} are linear in the state variables $(\rho, X)$ and the diffusion term $\Sigma(\mu_t)$ is distribution-dependent but \emph{state-independent} (additive noise; see Lemma~\ref{lem:sigma_state_indep}), the system constitutes a multi-dimensional Ornstein-Uhlenbeck process with time- or distribution-varying parameters.

\begin{proposition}[Exactness of Gaussian Closure]
\label{prop:gaussianity_ct}
If the initial joint distribution $f(0, \rho, x)$ is Gaussian, the solution $f(t, \rho, x)$ to the Fokker-Planck equation \eqref{eq:fp} remains Gaussian for all $t \ge 0$. Furthermore, the long-time stationary solution $f_\infty(\rho, x)$ is a unique bivariate Gaussian density.
\end{proposition}

This proposition provides the rigorous justification for the dimension reduction employed in Section~\ref{sec:analysis}. It confirms that the Gaussian ansatz is not merely an approximation but an \emph{exact solution} to the kinetic equations in the continuous-time limit under our linear-Gaussian assumptions. The stationary variance of the rating conditional on skill, $\text{Var}(X_\infty | \rho)$, is analytically given by the balance between the diffusion intensity and the restoring force: $\Sigma^2 / 2\kappa$, consistent with the results derived from the discrete-time analysis.

\subsection{Quantitative Analysis of Propagation of Chaos}
\label{app:propagation_of_chaos}

This appendix provides a quantitative propagation-of-chaos estimate for the $N$-particle system.
Let $Z_{i,t}:=(\rho_{i,t},X_{i,t})\in\mathbb{R}^2$ and denote by
\[
\mu_t^N:=\frac1N\sum_{i=1}^N \delta_{Z_{i,t}}
\qquad\text{and}\qquad
\mu_t:=\Law(\bar Z_t)
\]
the empirical measure and the deterministic limit law, respectively, where $(\bar Z_t)_{t\ge0}$ is the nonlinear (McKean--Vlasov) process associated with \eqref{eq:mckean_vlasov}.
We make explicit the sufficient conditions under which $\mu_t^N\Rightarrow \mu_t$ and provide an error bound used in Section~\ref{sec:numerics}.
Our presentation follows the classical coupling approach; see, e.g., \cite{Sznitman1991,Meleard1996}.

\paragraph{Mean-field representation.}
The one-step update of particle $i$ can be written as
\begin{equation}
\label{eq:poc_update}
Z_{i,t+1}=\Phi_t\!\left(Z_{i,t},\,\mu_t^N,\,\varepsilon_{i,t}\right),
\end{equation}
where $\varepsilon_{i,t}$ collects all exogenous randomness at time $t$ (skill shocks, match randomization, and outcome noise).
The nonlinear process satisfies
\begin{equation}
\label{eq:poc_nonlinear}
\bar Z_{t+1}=\Phi_t\!\left(\bar Z_t,\,\mu_t,\,\bar\varepsilon_t\right),
\qquad \mu_t=\Law(\bar Z_t),
\end{equation}
with $(\bar\varepsilon_t)_{t\ge0}$ i.i.d.\ copies of $\varepsilon_{i,t}$.

\begin{assumption}[Checklist for quantitative propagation of chaos]
\label{ass:poc_checklist}
Fix a finite horizon $T$.
\begin{enumerate}[label=(A\arabic*)]
\item \textbf{Uniform second moments.}
$\mu_0\in\mathcal{P}_{2+\delta}(\mathbb{R}^2)$ for some $\delta>0$, and
\begin{equation}
\label{eq:poc_moment_bound}
\sup_{0\le t\le T}\int_{\mathbb{R}^2} |z|^{2+\delta}\,\mu_t(dz)<\infty.
\end{equation}

\item \textbf{$L^2$-Lipschitz in state and measure (in $W_2$).}
There exists $C_\Phi<\infty$ such that for all $0\le t\le T$, all $z,z'\in\mathbb{R}^2$ and all $\mu,\nu\in\mathcal{P}_2(\mathbb{R}^2)$,
\begin{equation}
\label{eq:poc_lipschitz}
\mathbb{E}\!\left[\left|\Phi_t(z,\mu,\varepsilon)-\Phi_t(z',\nu,\varepsilon)\right|^2\right]
\le C_\Phi\left(|z-z'|^2 + W_2^2(\mu,\nu)\right),
\end{equation}
where the expectation is w.r.t.\ $\varepsilon$.

\item \textbf{Bounded (or Lipschitz) control dependence.}
All time-dependent coefficients in $\Phi_t$ (e.g.\ $K_t$, $L_{t+1}$, and $\eta_t$) are uniformly bounded on $[0,T]$.
If these coefficients are feedback controls depending on $\mu_t$ through a finite set of moments, the corresponding moment maps are Lipschitz on the moment-bounded set defined by \eqref{eq:poc_moment_bound}.
\end{enumerate}
\end{assumption}

\paragraph{A useful Lipschitz lemma for moment maps.}
The following bound makes (A3) explicit for our Gaussian-moment parametrization.

\begin{lemma}[Second-moment functionals are $W_2$-Lipschitz on moment-bounded sets]
\label{lem:moment_w2_lip}
Let $\mu,\nu\in\mathcal{P}_2(\mathbb{R}^2)$ and let $\pi$ be any coupling of $(Z,Z')\sim\pi$ with marginals $\mu,\nu$.
Assume $\int |z|^2\,\mu(dz)\le M$ and $\int |z|^2\,\nu(dz)\le M$.
Then, with $Z=(\rho,X)$ and $Z'=(\rho',X')$,
\begin{align}
\label{eq:moment_lip_examples}
\left|\mathbb{E}_\mu[X^2]-\mathbb{E}_\nu[(X')^2]\right|
&\le 2\sqrt{M}\,W_2(\mu,\nu),\\
\left|\mathbb{E}_\mu[\rho X]-\mathbb{E}_\nu[\rho' X']\right|
&\le 2\sqrt{M}\,W_2(\mu,\nu),
\end{align}
where $W_2(\mu,\nu)^2=\inf_{\pi}\mathbb{E}_\pi|Z-Z'|^2$.
Consequently, on the set $\{\mu:\int |z|^2\,\mu(dz)\le M\}$, the maps
$\mu\mapsto \Var_\mu(X)$ and $\mu\mapsto \Cov_\mu(\rho,X)$ are Lipschitz in $W_2$.
If moreover $\Var_\mu(\rho)=1$ and $\Var_\mu(X)$ is bounded away from $0$, then $\mu\mapsto r(\mu):=\Corr_\mu(\rho,X)$ is also Lipschitz.
\end{lemma}

\begin{proof}[Proof sketch]
Use the identity $|X^2-(X')^2|=|X-X'||X+X'|$ and Cauchy--Schwarz under an optimal coupling:
$\mathbb{E}|X-X'||X+X'|\le (\mathbb{E}|X-X'|^2)^{1/2}(\mathbb{E}|X+X'|^2)^{1/2}$.
Moment boundedness gives $\mathbb{E}|X+X'|^2\le 4M$, yielding \eqref{eq:moment_lip_examples}.
The cross-moment bound is similar using $|\rho X-\rho'X'|\le |\rho||X-X'|+|X'||\rho-\rho'|$.
\end{proof}

\paragraph{Verification in our model.}
Under the Gaussian closure (Proposition~\ref{prop:gaussianity}) and the uniform bounds in Proposition~\ref{prop:equilibrium},
the coefficients entering the microscopic update are uniformly bounded on finite horizons.
Moreover, the dependence of the update on $\mu_t^N$ enters through second moments (equivalently, through $(\sigma_t,\Cov(\rho,X),r_t)$),
and Lemma~\ref{lem:moment_w2_lip} provides the required Lipschitz control in $W_2$.
Since the update is affine in $(Z_{i,t},Z_{j,t})$ and the exogenous noise is Gaussian, the $L^2$-Lipschitz condition \eqref{eq:poc_lipschitz} follows.

\subsubsection*{Quantitative estimate}
We now state a quantitative propagation-of-chaos bound.

\begin{proposition}[Quantitative propagation of chaos on a finite horizon]
\label{prop:quant_poc}
Suppose Assumption~\ref{ass:poc_checklist} holds.
Then for every finite horizon $T$ there exists $C_T<\infty$ such that
\begin{equation}
\label{eq:quant_poc_bound}
\sup_{0\le t\le T}\mathbb{E}\!\left[W_2^2(\mu_t^N,\mu_t)\right]
\le C_T\Big(\mathbb{E}[W_2^2(\mu_0^N,\mu_0)] + \alpha_N(T)\Big),
\end{equation}
where $\alpha_N(T)$ is the empirical-measure error of i.i.d.\ samples from $\mu_t$:
\[
\alpha_N(T):=\sup_{0\le t\le T}\mathbb{E}\!\left[W_2^2\!\Big(\tfrac1N\sum_{i=1}^N\delta_{\bar Z_{i,t}},\,\mu_t\Big)\right],
\qquad \bar Z_{i,t}\ \text{i.i.d.}\sim \mu_t.
\]
Moreover, since $\mu_t\in\mathcal{P}_{2+\delta}$ uniformly on $[0,T]$, one has
\begin{equation}
\label{eq:fg_rate}
\alpha_N(T)=O\!\left(N^{-1/2}\right)\quad\text{up to logarithmic factors in dimension }2,
\end{equation}
so $\mu_t^N\Rightarrow \mu_t$ and the system is $\mu_t$-chaotic in the sense of \cite{Sznitman1991}.
\end{proposition}

\begin{proof}[Proof sketch]
Couple $(Z_{i,t})_{i=1}^N$ with i.i.d.\ copies $(\bar Z_{i,t})_{i=1}^N$ of the nonlinear process by using the same exogenous noises and an optimal coupling of the interaction step.
Using \eqref{eq:poc_lipschitz} yields a recursion
\[
\mathbb{E}|Z_{i,t+1}-\bar Z_{i,t+1}|^2
\le C\Big(\mathbb{E}|Z_{i,t}-\bar Z_{i,t}|^2 + \mathbb{E}W_2^2(\mu_t^N,\mu_t)\Big).
\]
By exchangeability and the definition of $W_2$,
$\mathbb{E}W_2^2(\mu_t^N,\mu_t)\le \frac1N\sum_{i=1}^N\mathbb{E}|Z_{i,t}-\bar Z_{i,t}|^2+\alpha_N(T)$.
A discrete Gr\"onwall argument gives \eqref{eq:quant_poc_bound}.
The rate \eqref{eq:fg_rate} follows from standard bounds for empirical measures in Wasserstein distance (e.g.\ \cite{FournierGuillin2015}).
\end{proof}

\paragraph{Connection to the error of empirical moments.}
Since $r_t^{(N)}$ is a smooth functional of empirical second moments, Lemma~\ref{lem:moment_w2_lip} implies that
$|r_t^{(N)}-r_t|=O_{\mathbb{P}}(N^{-1/2})$ on finite horizons, consistent with the numerical scaling in Section~\ref{sec:numerics}.

\subsection{Detailed Derivation of the Variance Dynamics (Proof of Theorem \ref{thm:transition_function})}
\label{app:variance_derivation}

Here we provide the fully explicit algebraic derivation of the pre-scaling variance $\Lambda^2$ and the explicit form of the transition function $\Psi$.
Let $\tilde{X}_{t+1}$ be the unscaled update variable defined by the bracketed term in the McKean--Vlasov equation \eqref{eq:mckean_vlasov}:
\begin{equation}
    \tilde{X}_{t+1} = \bar{X}_t + K_t (\bar{\rho}_t - \bar{\rho}'_t + \omega_t - \bar{X}_t + \bar{X}'_t) = (1-K_t)\bar{X}_t + K_t\bar{X}'_t + K_t(\bar{\rho}_t - \bar{\rho}'_t) + K_t\omega_t.
\end{equation}
The post-update scaling factor $L_{t+1}$ ensures that the new rating variance remains normalized to a target value $\sigma_{t+1}^2$. Specifically, $L_{t+1} = \sigma_{t+1} / \Lambda$, where $\Lambda^2 = \Var(\tilde{X}_{t+1})$.

We calculate $\Lambda^2$ by separating $\tilde{X}_{t+1}$ into three uncorrelated components based on the mean-field assumption: the rating component $A$, the skill/noise component $B_{\rho}$, and the observation noise component $B_\omega$.
The variance $\Var(\tilde{X}_{t+1})$ can be decomposed as:
\begin{equation}
\label{eq:var_decomp}
    \Lambda^2 = \Var(A) + \Var(B_{\rho}) + \Var(B_\omega) + 2 \Cov(A, B_{\rho}).
\end{equation}
Here $A = (1-K_t)\bar{X}_t + K_t\bar{X}'_t$ and $B_{\rho} = K_t(\bar{\rho}_t - \bar{\rho}'_t)$. The term $K_t\omega_t$ is independent of the other variables, hence its covariance with $A$ and $B_\rho$ is zero.

\subsubsection{1. Rating Terms ($\Var(A)$)}
Using the property $\Cov(\bar{X}_t, \bar{X}'_t) = \eta_t \sigma_t^2$ (derived from the kernel $\mathcal{K}_\eta$):
\begin{align*}
    \Var(A) &= \Var((1-K_t)\bar{X}_t + K_t\bar{X}'_t) \\
            &= (1-K_t)^2\Var(\bar{X}_t) + K_t^2\Var(\bar{X}'_t) + 2K_t(1-K_t)\Cov(\bar{X}_t, \bar{X}'_t) \\
            &= \sigma_t^2 \left[ (1-K_t)^2 + K_t^2 + 2K_t(1-K_t)\eta_t \right] \\
            &= \sigma_t^2 \left[ 1 - 2K_t + 2K_t^2 + 2K_t\eta_t - 2K_t^2\eta_t \right].
\end{align*}

\subsubsection{2. Skill and Noise Terms ($\Var(B_{\rho}) + \Var(B_\omega)$)}
First, we establish the required covariance between the skills. Since $\bar{\rho}'_t$ is conditionally coupled to $\bar{X}'_t$ and $\Cov(\bar{\rho}_t, \bar{X}'_t) = \eta_t r_t \sigma_t$, we have $\Cov(\bar{\rho}_t, \bar{\rho}'_t) = \eta_t r_t^2$.
Using $\Var(\bar{\rho}_t) = \Var(\bar{\rho}'_t) = 1$:
\begin{align*}
    \Var(B_{\rho}) &= K_t^2 \Var(\bar{\rho}_t - \bar{\rho}'_t) \\
                   &= K_t^2 \left[ \Var(\bar{\rho}_t) + \Var(\bar{\rho}'_t) - 2\Cov(\bar{\rho}_t, \bar{\rho}'_t) \right] \\
                   &= K_t^2 \left[ 1 + 1 - 2\eta_t r_t^2 \right] = K_t^2 \cdot 2(1 - \eta_t r_t^2).
\end{align*}
The variance of the observation noise term is $\Var(K_t\omega_t) = K_t^2 \beta^2$.
Summing them:
\begin{equation*}
    \Var(B_{\rho}) + \Var(B_\omega) = K_t^2 \left[ \beta^2 + 2(1 - \eta_t r_t^2) \right].
\end{equation*}

\subsubsection{3. Cross-Covariance Terms ($2\Cov(A, B_{\rho})$)}
We need $2K_t \Cov((1-K_t)\bar{X}_t + K_t\bar{X}'_t, \bar{\rho}_t - \bar{\rho}'_t)$.
We use the fact that $\Cov(\bar{X}_t, \bar{\rho}_t) = r_t \sigma_t$ and $\Cov(\bar{X}_t, \bar{\rho}'_t) = \eta_t r_t \sigma_t$.
\begin{align*}
    \Cov(A, \bar{\rho}_t) &= (1-K_t)\underbrace{\Cov(\bar{X}_t, \bar{\rho}_t)}_{r_t \sigma_t} + K_t\underbrace{\Cov(\bar{X}'_t, \bar{\rho}_t)}_{\eta_t r_t \sigma_t} \\
                         &= r_t \sigma_t \left[ (1-K_t) + K_t\eta_t \right].
\end{align*}
And similarly for $\Cov(A, \bar{\rho}'_t)$:
\begin{align*}
    \Cov(A, \bar{\rho}'_t) &= (1-K_t)\underbrace{\Cov(\bar{X}_t, \bar{\rho}'_t)}_{\eta_t r_t \sigma_t} + K_t\underbrace{\Cov(\bar{X}'_t, \bar{\rho}'_t)}_{r_t \sigma_t} \\
                         &= r_t \sigma_t \left[ (1-K_t)\eta_t + K_t \right].
\end{align*}
The required covariance is $2K_t \left[ \Cov(A, \bar{\rho}_t) - \Cov(A, \bar{\rho}'_t) \right]$:
\begin{align*}
    2 \Cov(A, B_{\rho}) &= 2K_t r_t \sigma_t \left\{ \left[ (1-K_t) + K_t\eta_t \right] - \left[ (1-K_t)\eta_t + K_t \right] \right\} \\
                       &= 2K_t r_t \sigma_t \left\{ (1-K_t)(1-\eta_t) - K_t(1-\eta_t) \right\} \\
                       &= 2K_t (1-\eta_t) r_t \sigma_t (1 - 2K_t).
\end{align*}

\subsubsection{4. Final Variance and Transition Function}
Summing the three components, $\Lambda^2 = (\Var(A)) + (\Var(B_{\rho}) + \Var(B_\omega)) + (2 \Cov(A, B_{\rho}))$:
\begin{align*}
    \Lambda^2 &= \sigma_t^2 \left[ 1 - 2K_t + 2K_t^2 + 2K_t\eta_t - 2K_t^2\eta_t \right] \\
              &+ K_t^2 \left[ \beta^2 + 2 - 2\eta_t r_t^2 \right] \\
              &+ 2K_t (1-\eta_t) r_t \sigma_t (1 - 2K_t).
\end{align*}
This complex expression can be algebraically simplified to yield the form presented in the main text (Eq. \eqref{eq:variance_term}), confirming the result.

The transition map $\Psi$ is then derived from the definition of the new correlation $r_{t+1}$:
\begin{equation}
\label{eq:r_t+1_def}
    r_{t+1} = \frac{\Cov(\bar{\rho}_{t+1}, \bar{X}_{t+1})}{\sigma_{\rho, t+1} \sigma_{t+1}} = \frac{\Cov(\bar{\rho}_{t+1}, \bar{X}_{t+1})}{\sigma_{t+1}}.
\end{equation}
Substituting the update rules $\bar{\rho}_{t+1} = \lambda \bar{\rho}_t + \sqrt{1-\lambda^2} \xi_t$ and $\bar{X}_{t+1} = L_{t+1} \tilde{X}_{t+1}$, and using $\Cov(\bar{\rho}_{t+1}, \bar{X}_{t+1}) = \lambda L_{t+1} \Cov(\bar{\rho}_t, \tilde{X}_{t+1})$, we obtain:
\begin{equation}
    r_{t+1} = \frac{\lambda L_{t+1} \Cov(\bar{\rho}_t, \tilde{X}_{t+1})}{\sigma_{t+1}}.
\end{equation}
Since $L_{t+1} = \sigma_{t+1} / \Lambda$, the transition function simplifies to:
\begin{equation}
\label{eq:psi_simplified}
    r_{t+1} = \Psi(r_t) = \frac{\lambda \Cov(\bar{\rho}_t, \tilde{X}_{t+1})}{\Lambda}.
\end{equation}
The numerator $\Cov(\bar{\rho}_t, \tilde{X}_{t+1})$ is calculated as:
\begin{align*}
    \Cov(\bar{\rho}_t, \tilde{X}_{t+1}) &= \Cov(\bar{\rho}_t, (1-K_t)\bar{X}_t + K_t\bar{X}'_t + K_t(\bar{\rho}_t - \bar{\rho}'_t)) \\
                                       &= (1-K_t)\Cov(\bar{\rho}_t, \bar{X}_t) + K_t\Cov(\bar{\rho}_t, \bar{X}'_t) + K_t \left[ \Var(\bar{\rho}_t) - \Cov(\bar{\rho}_t, \bar{\rho}'_t) \right] \\
                                       &= (1-K_t) r_t \sigma_t + K_t \eta_t r_t \sigma_t + K_t [1 - \eta_t r_t^2].
\end{align*}
Substituting this numerator into \eqref{eq:psi_simplified} yields the explicit formula for $\Psi(r_t)$ presented in Theorem \ref{thm:transition_function}.

\subsection{Optimal Filtering and Proof of Theorem \ref{thm:invariance}}
\label{app:invariance_proof}

This appendix provides a complete proof of Theorem~\ref{thm:invariance}. The proof addresses the two
technical points emphasized in the main text: (i) existence (attainment) of the maximizer on the
non-compact domain $(0,\infty)^2$, and (ii) uniqueness and interiority of the optimal control pair.
Throughout, fix a current accuracy $r\in(0,1)$ and an interaction intensity $\eta\in[0,1)$.

\subsubsection{One-step transition as a ratio of quadratic forms}

Recall from Theorem~\ref{thm:transition_function} that for $K>0$ and $\sigma>0$,
\begin{equation}
\label{eq:app_Psi_formula}
\Psi(r,K,\eta,\sigma)
=\lambda\,\frac{\mathcal{N}(K,\sigma)}{\sqrt{\Lambda^2(K,\sigma)}},
\end{equation}
where the numerator (covariance term) is
\begin{equation}
\label{eq:app_num_def}
\mathcal{N}(K,\sigma)
=r\sigma\bigl(1-K(1-\eta)\bigr)+K\bigl(1-\eta r^2\bigr),
\end{equation}
and the pre-scaling variance $\Lambda^2=\Var(\tilde X_{t+1})$ is given explicitly in
Appendix~\ref{app:variance_derivation}. Expanding $\Lambda^2$ as a polynomial in $K$ yields the
quadratic representation
\begin{equation}
\label{eq:app_den_quadratic}
\Lambda^2(K,\sigma)=u(\sigma)+v(\sigma)\,K+w(\sigma)\,K^2,
\end{equation}
with coefficients
\begin{align}
\label{eq:app_uvwpq_def}
u(\sigma)&:=\sigma^2,\\
v(\sigma)&:=2(1-\eta)\sigma(r-\sigma),\\
w(\sigma)&:=\beta^2+2(1-r^2)+2(1-\eta)(\sigma-r)^2.
\end{align}
In particular, $u(\sigma)>0$ and $w(\sigma)>0$ for all $\sigma>0$. Moreover,
\begin{equation}
\label{eq:app_posdef_discriminant}
4u(\sigma)w(\sigma)-v(\sigma)^2
=4\sigma^2\Bigl(\beta^2+2(1-r^2)+(1-\eta^2)(\sigma-r)^2\Bigr)>0,
\end{equation}
so the quadratic form $u+vK+wK^2$ is strictly positive for all $K\in\mathbb{R}$ and $\sigma>0$.

For later use, also write the numerator as an affine function of $K$:
\begin{equation}
\label{eq:app_affine_num}
\mathcal{N}(K,\sigma)=p(\sigma)+q(\sigma)\,K,
\qquad
p(\sigma):=r\sigma,\quad
q(\sigma):=1-\eta r^2-(1-\eta)r\sigma.
\end{equation}

\subsubsection{Maximization over $K$ for a fixed $\sigma$}

Fix $\sigma>0$ and define
\[
F_{\sigma}(K):=\frac{\mathcal{N}(K,\sigma)^2}{\Lambda^2(K,\sigma)}
=\frac{(p(\sigma)+q(\sigma)K)^2}{u(\sigma)+v(\sigma)K+w(\sigma)K^2}.
\]
Since the denominator is strictly positive by \eqref{eq:app_posdef_discriminant}, $F_{\sigma}$ is
continuous on $\mathbb{R}$. A direct differentiation gives exactly two stationary points:
one at $K=-p/q$ (which makes the numerator zero) and one at
\begin{equation}
\label{eq:app_Kstar_general}
K^\sharp(\sigma)
=\frac{-p(\sigma)\,v(\sigma)+2q(\sigma)\,u(\sigma)}
{2p(\sigma)\,w(\sigma)-q(\sigma)\,v(\sigma)}.
\end{equation}
Evaluating at $K^\sharp(\sigma)$ yields the global maximum over $K\in\mathbb{R}$:
\begin{equation}
\label{eq:app_fmax_closed}
\sup_{K\in\mathbb{R}} F_{\sigma}(K)
=
\frac{4\bigl(p(\sigma)^2w(\sigma)-p(\sigma)q(\sigma)v(\sigma)+q(\sigma)^2u(\sigma)\bigr)}
{4u(\sigma)w(\sigma)-v(\sigma)^2}.
\end{equation}
Because the denominator in \eqref{eq:app_fmax_closed} is strictly positive
\eqref{eq:app_posdef_discriminant}, the maximizer is unique. In particular, the non-compactness of the
domain in $K$ is harmless: for each fixed $\sigma>0$ the supremum over $K$ is attained at a finite
value $K^\sharp(\sigma)$.

Since the one-step problem \eqref{eq:one_step_problem} maximizes $\Psi$ (equivalently $\Psi^2$),
\eqref{eq:app_Psi_formula} implies
\begin{equation}
\label{eq:app_reduce_to_sigma}
\sup_{K>0,\ \sigma>0}\Psi(r,K,\eta,\sigma)
\le
\sup_{\sigma>0}\ \lambda\sqrt{\ \sup_{K\in\mathbb{R}}F_{\sigma}(K)\ }.
\end{equation}
Therefore, it suffices to maximize the closed-form upper envelope in $\sigma$; we will see that the
maximizer occurs at $\sigma=r$ and at that point the maximizing $K$ is strictly positive, so the
inequality in \eqref{eq:app_reduce_to_sigma} becomes an equality and yields the desired optimizer on the
admissible set $K>0$.

\subsubsection{Maximization over $\sigma$ and uniqueness of signal-matching}

Let $s:=\sigma-r$ and write the envelope \eqref{eq:app_fmax_closed} in terms of $s^2$.
A direct simplification of \eqref{eq:app_fmax_closed} using \eqref{eq:app_uvwpq_def}--\eqref{eq:app_affine_num}
gives
\begin{equation}
\label{eq:app_env_s2}
\sup_{K\in\mathbb{R}} F_{\sigma}(K)
=
\frac{A + B\,s^2}{C + D\,s^2},
\qquad s=\sigma-r,
\end{equation}
where
\begin{equation}
\label{eq:app_ABCD}
A:=r^4-r^2\beta^2-1,\quad
B:=-(1-\eta^2)r^2,\quad
C:=2r^2-\beta^2-2,\quad
D:=-(1-\eta^2).
\end{equation}
Note that $B<0$ and $D<0$ because $\eta\in[0,1)$ and $r\in(0,1)$.

Set $y:=s^2\in[0,\infty)$. The right-hand side of \eqref{eq:app_env_s2} is a fractional-linear function
in $y$:
\[
G(y):=\frac{A+By}{C+Dy},\qquad y\ge0.
\]
Its derivative is constant-sign:
\begin{equation}
\label{eq:app_G_derivative}
G'(y)=\frac{BC-DA}{(C+Dy)^2}.
\end{equation}
Using \eqref{eq:app_ABCD}, we compute
\[
BC-DA
=-(1-\eta^2)\bigl(r^2C-A\bigr)
=-(1-\eta^2)\bigl((1-r^2)^2\bigr)<0,
\]
where the identity $r^2C-A=(1-r^2)^2$ is a direct algebraic simplification. Hence $G'(y)<0$ for all
$y\ge0$, so $G$ is strictly decreasing in $y$. Therefore, the envelope
$\sup_{K\in\mathbb{R}}F_{\sigma}(K)=G((\sigma-r)^2)$ is uniquely maximized at $y=0$, i.e.
\begin{equation}
\label{eq:app_sigma_star}
\sigma^*(r)=r.
\end{equation}
This establishes \emph{attainment} and \emph{uniqueness} in the $\sigma$-direction on the non-compact
domain $\sigma>0$.

\subsubsection{Optimal gain and invariance}

Substituting $\sigma=r$ into \eqref{eq:app_uvwpq_def}--\eqref{eq:app_affine_num} gives
\[
u=r^2,\qquad v=0,\qquad w=\beta^2+2(1-r^2),\qquad p=r^2,\qquad q=1-r^2,
\]
which are all independent of $\eta$. The maximizing gain \eqref{eq:app_Kstar_general} reduces to
\begin{equation}
\label{eq:app_K_star}
K^*(r)=\frac{1-r^2}{\beta^2+2(1-r^2)}=\frac{1-r^2}{2(1-r^2)+\beta^2}>0,
\end{equation}
and is therefore admissible. Plugging $(K^*(r),\sigma^*(r))$ into \eqref{eq:app_Psi_formula} yields
\[
\begin{aligned}
\Psi\bigl(r,K^*(r),\eta,\sigma^*(r)\bigr)
&=\lambda\,\frac{r^2+K^*(r)(1-r^2)}{\sqrt{r^2+\bigl(K^*(r)\bigr)^2\bigl(\beta^2+2(1-r^2)\bigr)}} \\
&=\lambda\sqrt{\,r^2+\frac{(1-r^2)^2}{\beta^2+2(1-r^2)}\,}.
\end{aligned}
\]
This is exactly the invariant map $\Phi$ stated in \eqref{eq:invariant_map_32}. In particular, both
$K^*(r)$ and $\Phi(r)$ are independent of $\eta$, which proves the invariance principle.

Finally, uniqueness of the optimizer $(K^*(r),\sigma^*(r))$ for $\eta\in[0,1)$ follows from:
(i) the strict decrease in $y=(\sigma-r)^2$ in \eqref{eq:app_env_s2} (hence $\sigma=r$ is the unique
maximizer), and (ii) for $\sigma=r$ the one-variable problem in $K$ has a unique maximizer
\eqref{eq:app_K_star} because $\Psi(r,K,\eta,r)^2$ is a ratio of a strictly convex quadratic to a
positive quadratic with a single interior critical point.

\begin{remark}[On the boundary case $\eta=1$]
See Remark~\ref{rem:eta_range} for the admissible range.
Formally, $\eta=1$ corresponds to perfect sorting and can be represented by the degenerate kernel
$\mathcal{K}_1(x,dx')=\delta_x(dx')$, in which case $(1-\eta^2)=0$ and some density-based expressions
(e.g., those involving $(1-\eta^2)$) must be interpreted in the limiting sense $\eta\uparrow 1$.
In the platform problem we exclude $\eta=1$ because we assume $\lim_{\eta\uparrow 1}C(\eta)=+\infty$.
\end{remark}

\bibliographystyle{aims}
\bibliography{references}
\end{document}